\documentclass{article}

\usepackage{graphicx} 
\usepackage{amsfonts}
\usepackage{amssymb}
\usepackage{amsthm}
\usepackage{amsmath}
\usepackage{color}
\usepackage{url}

\usepackage{todonotes}

\newcommand{\Inv}[2]{{#2}^{(-#1)}}
\newcommand{\ELL}{\textnormal{ELL}}
\newcommand{\bfDelta}{\boldsymbol{\Delta}}
\DeclareMathOperator{\dom}{dom}
\DeclareMathOperator{\ran}{ran}

\renewcommand{\labelenumi}{(\arabic{enumi})}

\newcounter{casecounter}
\newtheorem{thm}{Theorem}[section]
\newtheorem{lem}[thm]{Lemma}
\newtheorem{prop}[thm]{Proposition}
\newtheorem{cor}[thm]{Corollary}
\newtheorem{case}[casecounter]{Case}
\theoremstyle{definition}

\newtheorem{defn}[thm]{Definition}

\newtheorem{question}[thm]{Question}

\newtheorem{remark}[thm]{Remark}
\newtheorem{fact}[thm]{Fact}

\begin{document}
	
	\title{Recovering a structure from its tree of tuples}
	
	\author{David Gonzalez, Matthew Harrison-Trainor\thanks{Harrison-Trainor was partially supported by a Sloan Research Fellowship and by the National Science Foundation under Grant DMS-2452105.}, and Julia Knight}
	
	\maketitle
	
	\begin{abstract}
		We show that for each countable ordinal $\alpha$, there is a $\bfDelta^0_\alpha$ operator that, for all $\mathcal{A}$ of Scott rank at most $\alpha$, takes copies of $T_\mathcal{A}$ to copies of $\mathcal{A}$, where $T_\mathcal{A}$ is the tree of tuples corresponding to $\mathcal{A}$. For finite $\alpha$, we show that this is best possible. The proof that there is a $\bfDelta^0_\alpha$ operator uses Harrington's method of workers. The sharpness result, for finite $n$, uses recent work of Harrison-Trainor on ``unfriendly jump inversion.''
	\end{abstract}

	\subsubsection*{AI Disclosure}
	
	{\small ChatGPT Sol 5.6 was used by the first and second author during the editing process and to search for mathematical and linguistic errors.}
	
	\section{Introduction}
	
	In their seminal paper \cite{FriedmanStanley}, Friedman and Stanley showed that the equivalence relation of isomorphism on trees is as complicated as the equivalence relation of isomorphism for structures in any finite relational language $\mathcal{L}$ by giving a Borel embedding that takes each countable $\mathcal{L}$-structure $\mathcal{A} \in Mod(\mathcal{L})$ to a tree $T_{\mathcal{A}}$ such that
	\[ \mathcal{A} \cong \mathcal{B} \Longleftrightarrow T_{\mathcal{A}} \cong T_{\mathcal{B}}.\] In fact, the embedding is Turing computable. 
	In \cite{HTMontalban}, $T_{\mathcal{A}}$, called the \textit{tree of tuples} of $\mathcal{A}$, is described as a labeled tree whose nodes are the finite tuples of $\mathcal{A}$, ordered by extension (tuples of length $n$ are at level $n$), with each tuple labeled by its atomic diagram. This construction has been used as a step in the proof of several further Borel complexity results, including the analysis of structures as diverse as linear orderings \cite{FriedmanStanley}, Boolean algebras \cite{CamerloGao}, and models of Peano arithmetic \cite{MR,GLRS}. 
	
	The tree of tuples construction does not have the regularity and preservation properties enjoyed by many of the commonly studied Borel embeddings. For example, by results of Gao \cite{Gao} (see also \cite{GonzalezRossegger} and \cite{HTKim}), the class $\mathbb{T}$ of labeled trees isomorphic to those in the range of the embedding is not Borel. In \cite{GonzalezKnight}, Gonzalez and Knight considered subclasses $\mathbb{T}^\alpha$ consisting of those $T\cong T_\mathcal{A}$ for $\mathcal{A}$ of Scott rank at most $\alpha$. Their result, contrasting with that of Gao, says that if $\alpha$ is an $X$-computable ordinal, then the class $\mathbb{T}^\alpha$ is $X$-effectively $\Pi^0_{2\alpha+2}$-complete. In \cite{robuster}, Montalb\'{a}n defined the \emph{Scott rank} of $\mathcal{A}$ to be the least ordinal $\alpha$ such that $\mathcal{A}$ has a $\Pi_{\alpha+1}$ Scott sentence, or equivalently, such that the orbits of all tuples in $\mathcal{A}$ are defined by $\Sigma_\alpha$-formulas (Montalb\'{a}n gave several further equivalent conditions).  This is the definition of Scott rank used in \cite{GonzalezKnight}. 
	
	In the present paper, we consider the following question: Given a labeled tree $T$ that is isomorphic to $T_{\mathcal{A}}$ for some $\mathcal{L}$-structure $\mathcal{A}$, how hard is it to compute a copy of $\mathcal{A}$? Harrison-Trainor and Montalb\'an showed in \cite{HTMontalban} that there is a structure $\mathcal{A}$ such that $T_\mathcal{A}$ has a computable copy, but $\mathcal{A}$ has no computable copy. Moreover, Harrison-Trainor and Montalb\'an showed\footnote{The published proof in 
		\cite{HTMontalban} is incorrect. See the erratum \cite{errata}.} that for each computable ordinal $\alpha\geq 2$, there is a structure $\mathcal{A}$ such that $T_\mathcal{A}$ has a computable copy, but $\mathcal{A}$ has no $\Delta^0_\alpha$ copy. Relativizing this to any oracle $X$, we see that there is no $\bfDelta^0_\alpha$ operator, for any ordinal $\alpha\geq 2$, which can uniformly recover structures from their tree of tuples.\footnote{By a $\bfDelta^0_\alpha$ operator we mean a $\bfDelta^0_\alpha$-measurable function between spaces of structures, and by an effective $\Delta^0_\alpha$ operator we mean a $\bfDelta^0_\alpha$ operator with a computable code.}
	
	The work of Gonzalez and Knight \cite{GonzalezKnight} mentioned above suggests that the level of difficulty in recovering a structure from its tree of tuples is connected to the Scott rank of the structure. It is quite easy to show that if $\mathcal{A}$ has Scott rank $1$, meaning that the automorphism orbit of each tuple is defined by an existential formula, then given a copy of $T_\mathcal{A}$, we can computably build a copy of $\mathcal{A}$. The proof is given in Case 1 of Section \ref{sec:mainproof}. On the other hand, the structure $\mathcal{A}$ constructed by Harrison-Trainor and Montalb\'an (such that $T_\mathcal{A}$ has a computable copy, while $\mathcal{A}$ does not) can easily be seen to have Scott rank $2$. A naive argument using the machinery of \cite{GonzalezKnight} shows that for each computable ordinal $\alpha = \beta+1$, there is an effective $\Delta^0_{2\beta+1}$ operator that, for all $\mathcal{A}$ of rank $\alpha$, takes copies of $T_\mathcal{A}$ to copies of $\mathcal{A}$. Using a worker construction, we prove something much stronger.
	
	\begin{thm}	\label{Bounding}
		
		For each computable ordinal $\alpha\geq 1$, there is an effective $\Delta^0_\alpha$ operator that, for each $\mathcal{A}\in Mod(\mathcal{L})$ of Scott rank at most $\alpha$ (or even \ELL-rank at most $\alpha$), takes copies of $T_\mathcal{A}$ to copies of~$\mathcal{A}$.
	\end{thm}
	
	\ELL-rank, which we define in Section 2.2, is a new variant of Scott rank based on the symmetric back-and-forth relations $\equiv_\alpha$ considered by Ehrenfeucht \cite{Ehrenfeucht} and L\"{a}uchli and Leonard \cite{Lauchli}. If a structure has Scott rank at most $\alpha$, then it also has \ELL-rank at most $\alpha$, but not vice versa, so the theorem stated in terms of \ELL-rank is stronger. 
	Theorem \ref{Bounding} relativizes. For each $X$ and each $X$-computable ordinal $\alpha\geq 1$, there is an $X$-effective $\Delta^0_\alpha$ operator that, for each $\mathcal{A}\in Mod(\mathcal{L})$ of \ELL-rank at most $\alpha$, takes copies of $T_\mathcal{A}$ to copies of $\mathcal{A}$. Thus, we get the following corollary.
	
	\begin{cor}
		For each countable ordinal $\alpha\geq 1$, there is a $\bfDelta^0_\alpha$ operator that, for each $\mathcal{A}\in Mod(\mathcal{L})$ of Scott rank at most $\alpha$ (or even \ELL-rank at most $\alpha$), takes copies of $T_\mathcal{A}$ to copies of~$\mathcal{A}$.
	\end{cor}
	
	We also give a partial hardness result saying that at finite levels, this cannot be improved.  For finite $n > 1$, there is no $\bfDelta^0_{n-1}$ operator that works even for structures of Scott rank at most $n$ (and hence no operator that works for structures of \ELL-rank at most $n$). This argument uses the unfriendly jump inversion of Harrison-Trainor \cite{unfriendly}.
	
	In the proof of Theorem \ref{Bounding}, we define objects like those in Ash's metatheorem. We have an alternating tree, a family of back-and-forth relations, etc., forming an $\alpha$-system. For us, $\alpha$ will be finite. We also have an instruction function, providing high-level information. What prevented us from simply applying Ash's metatheorem is the fact that these back-and-forth relations are not all c.e.\ relative to the appropriate oracle. The proof of Ash's metatheorem given in \cite{AK00} used workers,\footnote{ Ash's original proof of the metatheorem did not use workers; it was inspired by Martin's proof of Borel Determinacy. } and we use workers here. The new idea is that the instructions are changed as they are passed down. In consequence, Worker $\beta$ needs only to enumerate up to the $\beta$th back-and-forth relation. To make this work, we found it necessary to use the symmetric back-and-forth relations $\equiv_\beta$, and not the standard back-and-forth relations $\leq_\beta$. We believe that there may be other applications of this method. See Subsection \ref{rem:new} for further discussion. 
	
	\section{Preliminaries}
	
	\subsection{Back-and-forth relations} 
	
	\begin{defn}\label{defn:bfsym}
		
		Let $\mathcal{L}$ be a finite relational language, and let $\mathcal{A}\in Mod(\mathcal{L})$. Suppose $\bar{a}$ and $\bar{b}$ are tuples in $\mathcal{A}$ of the same length. 
		
		\begin{enumerate}
			
			\item $\bar{a}\equiv_0\bar{b}$ if $\bar{a}$ and $\bar{b}$ satisfy the same quantifier-free type,
			
			\item for $\alpha > 0$, $\bar{a}\equiv_\alpha\bar{b}$ if for all $\beta <\alpha$,
			\begin{itemize}
				\item for each $\bar{c}$ there exists $\bar{d}$ such that $\bar{a},\bar{c}\equiv_\beta\bar{b},\bar{d}$, and
				\item for each $\bar{d}$ there exists $\bar{c}$ such that $\bar{a},\bar{c}\equiv_\beta\bar{b},\bar{d}$.
			\end{itemize}
			
		\end{enumerate}
		
	\end{defn}
	
	The original proof of the Scott Isomorphism Theorem \cite{Scott} involved symmetric relations, although the added tuples $\bar{c}$ and $\bar{d}$ were single elements. Ehrenfeucht \cite{Ehrenfeucht} and L\"{a}uchli and Leonard \cite{Lauchli} used tuples. Their equivalence relations were exactly as above. 
	
	Often in computable structure theory, one uses instead the asymmetric relations defined as follows.
	\begin{defn}
		
		Let $\mathcal{L}$ be a finite relational language, and let $\mathcal{A}\in Mod(\mathcal{L})$. Suppose $\bar{a}$ and $\bar{b}$ are tuples in $\mathcal{A}$ of the same length. 
		
		\begin{enumerate}
			
			\item $\bar{a}\leq_0\bar{b}$ if $\bar{a}$ and $\bar{b}$ satisfy the same quantifier-free type,
			
			\item for $\alpha > 0$, $\bar{a}\leq_\alpha\bar{b}$ if for all $\beta <\alpha$ and $\bar{c}$ there exists $\bar{d}$ such that $\bar{b},\bar{c}\leq_\beta\bar{a},\bar{d}$.
			
		\end{enumerate}
		
	\end{defn}
	
	What makes these asymmetric relations useful is the fact, proved by Karp \cite{Karp}, that $\bar{a} \leq_\alpha \bar{b}$ if and only if every $\Sigma_\alpha$ formula true of $\bar{b}$ is true of $\bar{a}$. Often one will see $\equiv_\alpha$ defined by $\bar{a} \equiv_\alpha \bar{b}$ if and only if $\bar{a} \leq_\alpha \bar{b}$ and $\bar{b} \leq_\alpha \bar{a}$. These equivalence relations are quite different from the ones in Definition \ref{defn:bfsym}. When we use the notation in this paper, we always intend the meaning in Definition 2.1, never the other possible meaning. We are describing the \textit{true} symmetric back-and-forth relations, rather than the symmetrization of the asymmetric relations. The precise reasons behind our use of the symmetric relations are somewhat technical, but we offer at least some explanation in Remark \ref{rem:symm} and Subsection \ref{rem:new}. Our symmetric relations still satisfy the following.
	
	\begin{lem}
		
		If $\bar{a}\equiv_\beta\bar{a}'$, then $\bar{a}\leq_\beta\bar{a}'$ and $\bar{a}'\leq_\beta\bar{a}$.
		
	\end{lem}
	
	\begin{proof}
		
		This is a straightforward induction on $\beta$.
	\end{proof}
	
	\begin{remark}
		For $X$-computable ordinals $\beta$, there are $X$-computable $\Pi_{2\beta}$ formulas that, in all $\mathcal{L}$-structures, define the relations $\leq_\beta$. The same is true for $\equiv_\beta$. In \cite{ChenGonzalezHT}, it is shown that there is no simpler definition of the relations $\leq_\beta$, and, although we do not have a proof, we believe that this is true for the relations $\equiv_\beta$ as well.
	\end{remark}
	
	\begin{remark}
		We often define the back-and-forth relations $\bar{a}\leq_\alpha\bar{b}$ or $\bar{a}\equiv_\alpha\bar{b}$ between tuples $\bar{a}$ and $\bar{b}$ of different lengths, with $\bar{b}$ being of greater length than $\bar{a}$. This is just saying that the relation holds between $\bar{a}$ and the initial segment of $\bar{b}$ of the same length as $\bar{a}$.
	\end{remark}
	
	\subsection{Scott rank and \ELL-rank}
	
	While historically there have been numerous inequivalent definitions of Scott rank, in this paper we use Montalb\'{a}n's Scott rank introduced in \cite{robuster}. This rank is robust because, unlike the other definitions, it can be equivalently defined in several different ways. Montalb\'{a}n demonstrated the equivalence of several conditions. Below, we give just three of these conditions. 
	
	\begin{thm} [\cite{robuster}]
		
		For a structure $\mathcal{A}\in Mod(\mathcal{L})$ and a countable ordinal $\alpha\geq 1$, the following are equivalent: 
		
		\begin{enumerate}
			
			\item $\mathcal{A}$ has a $\Pi_{\alpha+1}$ Scott sentence,
			
			\item for all tuples $\bar{a}$ in $\mathcal{A}$, the orbit is defined by a $\Sigma_\alpha$ formula,
			
			\item for each tuple $\bar{a}$, there is a $\Sigma_\alpha$ formula true of $\bar{a}$ that (in $\mathcal{A}$) implies all $\Pi_\alpha$ formulas true of $\bar{a}$ in $\mathcal{A}$. 
			
		\end{enumerate}
		
	\end{thm}
	
	\begin{defn}
		
		The \emph{Scott rank} of $\mathcal{A}$ is the least ordinal $\alpha$ satisfying the conditions of Montalb\'{a}n's Theorem.
		
	\end{defn}
	
	The following observation is given in \cite{GonzalezKnight}.
	
	\begin{prop}
		
		$\mathcal{A}$ has Scott rank $\alpha$ iff $\alpha$ is least such that the following equivalent conditions hold:
		
		\begin{enumerate}
			
			\item for each $\bar{a}$, there exist $\beta < \alpha$ and $\bar{b}$ such that for all $\bar{a}',\bar{b}'$, $\bar{a},\bar{b}\leq_\beta \bar{a}',\bar{b}'$ implies $\bar{a}\leq_\alpha\bar{a}'$, 
			
			\item for each $\bar{a}$, there exist $\beta < \alpha$ and $\bar{b}$ such that for all 
			$\bar{a}',\bar{b}'$, $\bar{a},\bar{b}\leq_\beta\bar{a}',\bar{b}'$ implies $\bar{a}\cong\bar{a}'$. 
			
		\end{enumerate}
		
	\end{prop}
	
	For the symmetric relations $\equiv_\beta$, we have the analogous result.
	
	\begin{prop}
		\label{prop5}
		
		For a countable ordinal $\alpha$ and $\mathcal{A}\in Mod(\mathcal{L})$, the following are equivalent:
		
		\begin{enumerate}
			
			\item for each $\bar{a}$, there exist $\beta < \alpha$ and $\bar{b}$ such that for all $\bar{a}'$ and 
			$\bar{b}'$, $\bar{a},\bar{b}\equiv_\beta\bar{a}',\bar{b}'$ implies $\bar{a}\equiv_\alpha\bar{a}'$,
			
			\item for each $\bar{a}$, there exist $\beta < \alpha$ and $\bar{b}$ such that for all $\bar{a}'$ and 
			$\bar{b}'$, $\bar{a},\bar{b}\equiv_\beta\bar{a}',\bar{b}'$ implies $\bar{a}\cong\bar{a}'$.
			
		\end{enumerate}
		
	\end{prop}
	
	\begin{proof}
		
		It is clear that (2) implies (1). Supposing (1), we prove (2). It is enough to show that the family of functions $\bar{a}\rightarrow\bar{a}'$ such that $\bar{a}\equiv_\alpha\bar{a}'$ has the back-and-forth property. Suppose $\bar{a}\equiv_\alpha\bar{a}'$. Given $c$, we need $c'$ such that $\bar{a},c\equiv_\alpha \bar{a}',c'$. By assumption, there exist $\beta < \alpha$ and $\bar{b}$ such that $\bar{a},c,\bar{b}\equiv_\beta\bar{a}',c',\bar{b}'$ implies $\bar{a},c\equiv_\alpha\bar{a}',c'$. Since $\bar{a}\equiv_\alpha\bar{a}'$, there exist $c'$ and $\bar{b}'$ such that $\bar{a},c,\bar{b}\equiv_\beta\bar{a}',c',\bar{b}'$. Then $\bar{a},c\equiv_\alpha\bar{a}',c'$, as required. We also need to show that given $c'$, there exists $c$ such that $\bar{a},c\equiv_\alpha \bar{a}',c'$, but this follows from the symmetry of the relation~$\equiv_\alpha$.
	\end{proof}
	
	Here is the modified rank.
	
	\begin{defn}
		
		An $\mathcal{L}$-structure $\mathcal{A}$ has \emph{\ELL-rank $\alpha$} if $\alpha$ is least such that the conditions in Proposition \ref{prop5} hold. 
		
	\end{defn}
	
	\begin{remark}
		\label{rem:symm}
		
		This differs slightly from the definition saying that the rank of an $\mathcal{L}$-structure $\mathcal{A}$ is the least $\alpha$ such that for all tuples $\bar{a},\bar{b}$ we have
		\[ \bar{a}\equiv_\alpha \bar{b} \Longrightarrow \bar{a} \equiv_{\alpha + 1} \bar{b}.\]
		That notion of Scott rank is the one called $SR(\mathcal{A})$ in the book by Ash and Knight \cite{AK00}. We use \ELL-rank because it is precisely the rank that arises from the proof of Theorem \ref{Bounding} and, hence, gives the optimal result.
		
	\end{remark}
	
	\subsection{The tree of tuples}
	
	Let $\mathcal{L}$ be a relational language. For simplicity, we suppose that it is finite.\footnote{To justify this, we may apply a result of Lavrov \cite{Lavrov} saying that for any countable language $\mathcal{L}$, $Mod(\mathcal{L})$ is Borel reducible to the class of undirected graphs, and if $\mathcal{L}$ is computable, then the reduction is Turing computable.}
We include in our language the logical constant $\top$ (truth). Let $\mathcal{D}$ be the set of $\mathcal{L}$-formulas $D(x_1,\ldots,x_n)$ formed as the conjunction of formulas saying either $R(\bar{x})$ or $\neg{R(\bar{x})}$, for each relation symbol $R\in \mathcal{L}$ and each tuple $\bar{x}$ from among $x_1,\ldots,x_n$. We include in $\mathcal{D}$ the formula $\top$ (truth), with the empty sequence of variables. To each $\mathcal{A}\in Mod(\mathcal{L})$, we assign a labeled tree $T_\mathcal{A}$, the \emph{tree of tuples}. The language of the labeled trees consists of a unary function symbol $p$, for the predecessor function (which has the convention $p(root)=root$), and unary relation symbols $U_D$ for $D\in\mathcal{D}$, for the labels. The tree itself will be isomorphic to $\omega^{<\omega}$. Each node $\sigma$ of the tree $T_\mathcal{A}$ \emph{represents} a tuple $a_1,\ldots,a_n$ of distinct elements of $\mathcal{A}$, with each non-empty tuple showing up infinitely many times. The node $\sigma$ is in a single unary relation $U_D$, where $\mathcal{A}\models D(a_1,\ldots,a_n)$.

\bigskip

	\begin{defn}
		
		For $\mathcal{A}\in Mod(\mathcal{L})$, the nodes of $T_\mathcal{A}$ are defined by induction on the level as follows:
		
		\begin{enumerate}
			
			\item At level $0$, the root node, representing the empty tuple, has the label $U_\top$.
			
			\item Suppose $\sigma$ is a node at level $n$, representing an $n$-tuple $\bar{a}$ from $\mathcal{A}$, and with label $U_D$, where $D$ gives the atomic type of $\bar{a}$. Let $a'$ be an element of $\mathcal{A}$ not in $\bar{a}$. Then $\sigma$ has infinitely many successors $\sigma'$ representing the tuple $\bar{a},a'$, with the appropriate label $U_{D'}$, where $D'(x_1,\ldots,x_n,x_{n+1})$ gives the atomic type of the $(n+1)$-tuple $\bar{a},a'$. 
			
		\end{enumerate} 
		
	\end{defn}
	
	The tree of tuples captures many aspects of the structure, the most important being the back-and-forth structure. Indeed, this is the reason that
	\[ T_{\mathcal{A}} \cong T_{\mathcal{B}} \Longrightarrow \mathcal{A} \cong \mathcal{B}.\]
	If $T_{\mathcal{A}} \cong T_{\mathcal{B}}$ then using a back-and-forth procedure, we can arrive at an isomorphism between a copy of $\mathcal{A}$ and a copy of $\mathcal{B}$, each represented by a path through the appropriate tree. We make the following definition to capture the back-and-forth relations on the tree $T_{\mathcal{A}}$. 
	In the following definition, we use the notation $\sigma\succeq\tau$ for $\sigma,\tau \in T_{\mathcal{A}}$ to mean that $p^n(\sigma)=\tau$ for some $n$.
	In $T_{\mathcal{A}}$, the relation $\sigma\succeq\tau$ is defined by a computable $\Sigma_1$ formula. This shows that the relation $\sigma\succeq\tau$ is uniformly c.e.\ relative to the tree. In fact, it is computable relative to the tree. To see this, note that $\sigma\succeq\tau$ iff $\tau$ is among the finite set of predecessors of $\sigma$, which we can effectively list. 
		
	\begin{defn}
		 If $\sigma$ is at level $n$ in $T_\mathcal{A}$, we say that $x_1,\ldots,x_n$ are the variables of $\sigma$.  Let $\sigma,\tau \in T_{\mathcal{A}}$, and let $\bar{x},\bar{y}$ be tuples of variables of the same length, where $\bar{x}$ is among the variables of $\sigma$ and $\bar{y}$ is among the variables of $\tau$.  Then 
		
				\begin{enumerate}
			\item $(\sigma,\bar{x}) \equiv_0 (\tau,\bar{y})$ if $\sigma$ assigns to $\bar{x}$ the same quantifier-free type that $\tau$ assigns to $\bar{y}$.
			\item for $\alpha > 0$, $(\sigma,\bar{x}) \equiv_\alpha (\tau,\bar{y})$ if for all $\beta < \alpha$,
			\begin{itemize}
				\item for each $\sigma' \succeq \sigma$ and each $\bar{x}'$  (among the variables of $\sigma'$), there exist $\tau' \succeq \tau$ and $\bar{y}'$ (among the variables of $\tau'$) such that $(\sigma',\bar{x}\bar{x}') \equiv_\beta (\tau',\bar{y}\bar{y}')$, and
				\item for each $\tau' \succeq \tau$ and each $\bar{y}'$ (among the variables of $\tau'$), there exist $\sigma' \succeq \sigma$ and $\bar{x}'$ (among the variables of $\sigma'$) such that $(\sigma',\bar{x}\bar{x}') \equiv_\beta (\tau',\bar{y}\bar{y}')$.
			\end{itemize}
		\end{enumerate}
		We can make a similar definition for the asymmetric back-and-forth relations. Like the back-and-forth relations on $\mathcal{A}$ itself, the relations $\equiv_\alpha$ on $T_\mathcal{A}$ are $\Pi^0_{2\alpha}$.
	\end{defn}

	In $T_{\mathcal{A}}$, each node represents a tuple from $\mathcal{A}$.  Suppose $\sigma$ represents $(a_1,\ldots,a_r)$ and $\tau$ represents $(b_1,\ldots,b_s)$  ---we have that \[(\sigma,(x_{i_1},\ldots,x_{i_n})) \equiv_\alpha (\tau,(y_{j_1},\ldots,y_{j_n}))\]
	if and only if $a_{i_1},\ldots,a_{i_n} \equiv_\alpha b_{j_1},\ldots,b_{j_n}$. The advantage of the definition above is that it is independent of the identification of nodes in $T_{\mathcal{A}}$ with tuples of $\mathcal{A}$ and hence is independent of the copy of $\mathcal{A}$.
	
	\subsection{Generic paths}
	
	For an infinite $\mathcal{A}\in Mod(\mathcal{L})$, and $\sigma$ at level $n$ of $T_\mathcal{A}$, the label assigns a quantifier-free type to the variables $x_1,\ldots,x_n$, making all of the elements distinct. Each path $p$ through $T$ yields a structure $\mathcal{A}^p$, with universe equal to the full set of variables $\{x_1,x_2,\ldots\ \}$. The structure $\mathcal{A}^p$ is always isomorphic to a substructure of $\mathcal{A}$, and, in fact, we get copies of all infinite substructures of $\mathcal{A}$ this way. Each node $\sigma$ extends to a path $p$ such that $\mathcal{A}^p\cong\mathcal{A}$. We want, just looking at $T\cong T_\mathcal{A}$, to construct a path $p$ that we can be sure will yield a copy of $\mathcal{A}$. 
	
	In \cite{HTMontalban}, it is noted that a sufficiently generic path will yield a copy of $\mathcal{A}$. 
	In \cite{GonzalezKnight}, we define a forcing notion and show that if $\mathcal{A}$ has a $\Pi_\alpha$ Scott sentence, and $p$ is a generic path through $T_\mathcal{A}$, then $\mathcal{A}^p\cong\mathcal{A}$. Also in \cite{GonzalezKnight}, we define, without reference to forcing, what it means for a path $p$ to be ``$\alpha$-generic,'' and we show that for $\mathcal{A}$ of Scott rank at most $\alpha$, if $p$ is an $\alpha$-generic path through $T_\mathcal{A}$, then $\mathcal{A}^p\cong\mathcal{A}$. 
	We will define what it means for a path to be \ELL-$\alpha$-generic, and we show that for $\mathcal{A}$ of \ELL-rank at most $\alpha$, if $p$ is an \ELL-$\alpha$-generic path through $T_\mathcal{A}$, then $\mathcal{A}^p\cong\mathcal{A}$. 
	We begin with the definition of $\alpha$-genericity from \cite{GonzalezKnight}. 
	
	\begin{defn}
		Let $T \cong T_\mathcal{A}$, and let $\alpha\geq 1$ be a countable ordinal. A path $p$ through $T$ is \emph{$\alpha$-generic} provided that for all $\beta < \alpha$ and all pairs of nodes $(\nu,\tau)$ such that $\nu\prec p$ and $\nu\prec\tau$, where $\nu$ has variables $\bar{x}$, and $\tau$ has variables $\bar{x},\bar{y}$, there is some $\tau'$, with variables including $\bar{x}$ and some $\bar{z}$, such that $\nu\prec\tau'\prec p$ and $(\tau,\bar{x},\bar{y})\leq_\beta (\tau',\bar{x},\bar{z})$. 
	\end{defn}
	
	Here is the analogous definition, with $\equiv_\beta$ replacing $\leq_\beta$. 
	
	\begin{defn}
		
		Let $T\cong T_\mathcal{A}$, and let $\alpha\geq 1$ be a countable ordinal. A path $p$ through $T$ is 
		\emph{\ELL-$\alpha$-generic} provided that for all $\beta < \alpha$ and all pairs of nodes $(\nu,\tau)$ such that $\nu\prec p$ and $\nu\prec\tau$, where $\nu$ has variables $\bar{x}$, and $\tau$ has variables $\bar{x},\bar{y}$, there is some $\tau'$, with variables including $\bar{x}$ and some $\bar{z}$, such that $\nu\prec\tau'\prec p$ and $(\tau,\bar{x},\bar{y})\equiv_\beta (\tau',\bar{x},\bar{z})$. 
		
	\end{defn}
	
	Here is the analogue of the result from \cite{GonzalezKnight}.
	
	\begin{prop}
		\label{genericpath}
		
		Suppose $\mathcal{A}\in Mod(\mathcal{L})$ has \ELL-rank at most $\alpha$, and let $p$ be an \ELL-$\alpha$-generic path through $T \cong T_\mathcal{A}$. Then $\mathcal{A}^p\cong\mathcal{A}$. 
		
	\end{prop}
	
	\begin{proof}
		
		Since the notion of an \ELL-$\alpha$-generic path is invariant under isomorphisms, we may assume that $T = T_{\mathcal{A}}$. Let $\mathcal{F}$ be the set of finite partial isomorphisms $f$ such that for some $\sigma\prec p$, with variables $\bar{x}$, $f$ maps $\bar{x}$ to $\bar{a}\in\mathcal{A}$ and $\sigma$ represents a tuple in the orbit of $\bar{a}$. We show that $\mathcal{F}$ has the back-and-forth property. Clearly, $\mathcal{F}$ is non-empty. 
		
		\medskip 
		\noindent
		\textbf{Forth}: Suppose $\bar{x}\rightarrow\bar{a}$ is in $\mathcal{F}$. Say $\sigma\prec p$ has variables $\bar{x}$ and represents a tuple $\bar{a}'$ such that $\bar{a}'\cong \bar{a}$. Let $y$ be a variable not in $\bar{x}$---for simplicity, let it be the next variable. Let $\sigma'$ be the successor of $\sigma$ with $\sigma'\prec p$. Then $\sigma'$ represents some $\bar{a}',c'$, and the automorphism taking $\bar{a}$ to $\bar{a}'$ takes some $c$ to $c'$. Then $\bar{x},y\rightarrow \bar{a},c$ is in $\mathcal{F}$. 
		
		\medskip
		\noindent
		\textbf{Back}: Suppose $f\in\mathcal{F}$, where $f$ takes $\bar{x}$ to $\bar{a}$. Say $\sigma\prec p$ has variables $\bar{x}$ and represents a tuple $\bar{a}'$ in the orbit of $\bar{a}$. Let $c$ be an element of $\mathcal{A}$ not in $\bar{a}$. We have $c'$ such that $\bar{a},c\cong\bar{a}',c'$. There exist $\bar{b}'$ and $\beta < \alpha$ such that $\bar{a}',c',\bar{b}'\equiv_\beta \bar{a}'',c'',\bar{b}''$ implies $\bar{a}',c'\cong \bar{a}'',c''$. Take $\tau$, with variables $\bar{x},y,\bar{u}$ such that $\sigma\prec\tau$, and $\tau$ represents $\bar{a}',c',\bar{b}'$. By \ELL-$\alpha$-genericity, there exists $\tau'$ with $\sigma \prec \tau'\prec p$ with variables including $\bar{x}$ and some $z$ and $\bar{v}$ such that $(\tau,\bar{x},y,\bar{u})\equiv_\beta (\tau',\bar{x},z,\bar{v})$. Then $\tau'$ represents a tuple that includes $\bar{a}'$ and some $c'',\bar{b}''$ corresponding to $z$ and $\bar{v}$, and $\bar{a},c\cong\bar{a}',c''$. Therefore, the function $f'$ taking $\bar{x},z$ to $\bar{a},c$ is in $\mathcal{F}$. 
	\end{proof} 
	
	\begin{remark}\label{rmk:getGeneric} 
		By examining the alternations of quantifiers present in the definition of \ELL-$\alpha$-generic, we can bound the complexity of an operator that, given $T\in\mathbb{T}$, produces an $\alpha$-generic path. If $\alpha$ is an $X$-computable limit ordinal, then we may take the operator to be $X$-effective $\Delta^0_\alpha$, and if $\alpha = \beta+1$ is an $X$-computable successor ordinal, then we may take the operator to be $X$-effective $\Delta^0_{2\beta+1}$. 
	\end{remark}
	
	\section{Proof of Theorem \ref{Bounding}} \label{sec:mainproof}
	
	In this section, we prove Theorem \ref{Bounding}. We recall the statement below. 
	
	\bigskip
	\noindent
	\textbf{Theorem \ref{Bounding}.} Let $X\subseteq\omega$, and suppose $\alpha$ is an $X$-computable ordinal. Then there is an $X$-effective $\Delta^0_\alpha$ operator that, for structures $\mathcal{A}$ of \ELL-rank at most $\alpha$, takes copies of $T_\mathcal{A}$ to copies of $\mathcal{A}$.

	\bigskip
	
	We split the proof of Theorem \ref{Bounding} into the following cases: 
	
	\begin{enumerate}
		
		\item $\alpha = 1$, 
		
		\item $\alpha = n > 1$ finite, 
		
		\item $\alpha$ is a limit ordinal, 
		
		\item $\alpha = \lambda + n$, where $\lambda$ is a limit ordinal and $n\geq 1$ is finite. 
		
	\end{enumerate}
	
	We begin with the case $\alpha = 1$, which is relatively simple but illustrates the general method. We then follow with the case of $\alpha = n > 1$ finite, which is the key case. The case of $\alpha$ a limit ordinal is easy, since $2\alpha = \alpha$. Finally, the method for finite $n$ can easily be adapted to the case $\alpha = \lambda + n$.
	
	\subsection{The case $\alpha = 1$}
	
	In this subsection, we prove Theorem 1.1 in Case 1, where $\alpha = 1$.
	
	\begin{case}
		There is a Turing computable operator $\Phi$ such that if $T\cong T_\mathcal{A}$ and $\mathcal{A}$ is of \ELL-rank at most $1$, then $\Phi(T)\cong \mathcal{A}$.
	\end{case}
	\begin{proof}
		
		It is enough to give an effective procedure that, given such a $T$, produces an \ELL-$1$-generic path $p$. Say that $T\cong T_\mathcal{A}$, where $\mathcal{A}$ has \ELL-rank at most $1$. By Proposition \ref{genericpath}, $\mathcal{A}^p\cong\mathcal{A}$. It is clear that we can pass effectively from $p$ to the atomic diagram of $\mathcal{A}^p$---the label on the node at level $n$ gives the atomic type of $(x_1,\ldots,x_n)$.
		
		Computably in $T$, we can list the pairs $(\nu,\tau)$ such that $\nu\prec\tau$. We satisfy the following requirements:
		\begin{enumerate}
			\item[\textbf{$R_n$}:] For the $n^{th}$ pair $(\nu,\tau)$ with $\nu\prec\tau$, where $\nu$ has variables $\bar{x}$ and $\tau$ has variables $\bar{x},\bar{y}$, and $\tau$ assigns the atomic type $D(\bar{x},\bar{y})$, if $\nu\prec p$, then there exists $\tau'$, with variables including $\bar{x}$ and some $\bar{z}$, such that $\nu\prec\tau'\prec p$ and $(\tau,\bar{x},\bar{y})\equiv_0(\tau',\bar{x},\bar{z})$. 
		\end{enumerate} 
		
		At each stage $s$, we have determined a node $\sigma_s$ on $p$. We satisfy the first requirement that needs attention. At stage $s+1$, consider requirements $R_n$ for $n\leq s+1$, and $(\nu,\tau)$ among the first $s$ on our list. The requirement needs attention if $\nu\prec\sigma_s$, and there is no $\tau'\prec \sigma_s$ witnessing the appropriate atomic type. If we do not find a requirement needing attention, then $\sigma_{s+1}$ is just some successor of $\sigma_s$. Assuming that there is a requirement $R_n$ needing attention, we take the first. There is some extension $\tau''$ of $\sigma_s$ satisfying the requirement, since $\tau$ and $\sigma_s$ extend to paths that yield copies of $\mathcal{A}$. Knowing that there is a node satisfying $\tau''$, we can effectively find one, and we let this be $\sigma_{s+1}$. The resulting path $p$ satisfies the witnessing requirements, and is therefore \ELL-$1$-generic.
	\end{proof} 
	
	\subsection{The case $\alpha$ finite, greater than $1$} 
	
	In this subsection, we prove Theorem 1.1 in Case 2, where $\alpha$ is finite, greater than $1$. 
	
	\begin{case}
		
		For each finite $n > 1$, there is an effective $\Delta^0_n$ operator that, for all $\mathcal{A}$ of \ELL-rank at most $n$, maps copies of $T_\mathcal{A}$ to copies of $\mathcal{A}$. 
	\end{case}
	
	\begin{proof}
		
		We describe an effective $\Delta^0_n$ procedure that, for each $T\cong T_{\mathcal{A}}$ with $\mathcal{A}$ of \ELL-rank at most $n$, gives a structure $\mathcal{B}$ (with universe $\omega$) that is isomorphic to $\mathcal{A}^p$ for some \ELL-$n$-generic path $p$. We use a worker construction, with $n$ workers, where for $1\leq k\leq n$, Worker $k$ has oracle $\Delta^0_k(\Delta^0_n(T)) = \Delta^0_{n + k-1}(T)$. In describing the construction, we suppress the oracle $T$. 
		We have already defined $\equiv_k$ on pairs $(\sigma,\bar{x})$, where $\sigma$ is a node in $T$ and $\bar{x}$ is among the variables of $\sigma$. As we will see in Fact \ref{fact} below, for $1\leq k \leq n$, Worker $k$ can compute $\equiv_{k-1}$, and for $1\leq k \leq n-2$, Worker $k$ can also compute $\equiv_k$. 
		
		\subsubsection{Workers and the metatheorem of Ash}
		\label{rem:new}
		
		Ash proved a metatheorem, with conditions on some abstract objects, guaranteeing the existence of a c.e.\ set that satisfies some high-level requirements. There are a number of applications of the metatheorem in computable structure theory. Although the metatheorem does not require it, most applications use the standard asymmetric back-and-forth relations. While the original proof of the metatheorem did not use workers, the proof in \cite{AK00} does. We say a little about the proof from \cite{AK00}, in the case where there are just $n$ workers, since our proof will follow a similar outline. For $1\leq k\leq n$, Worker $k$ uses the oracle $\Delta^0_k$; throughout, the construction may be relativized. Worker $n$ determines an infinite sequence $\ell_0u^n_1\ell^n_1u^n_2\ell^n_2\ldots\ $ that is a path through a c.e.\ ``alternating'' tree $P$, where the terms $u^n_s$ are given by a $\Delta^0_n$ ``instruction'' function. There is a computable ``enumeration'' function $E$ that, for each possible $\ell$, enumerates a set. The goal of the construction is for $\cup_s E(\ell^n_s)$ to be c.e. 
		
		For $1\leq k < n$, Worker $k$ produces a sequence, based on guesses at the sequence produced by Worker $k+1$. The sequence $\ell^n_s$ produced by Worker $n$ is a subsequence of that produced by Worker $1$. For $k\leq n$, Worker $k$ preserves $\leq_{k-1}$; that is, $\ell^k_m\leq_{k-1}\ell^k_{m+1}$. Since $\ell\leq_k\ell'$ implies $\ell\leq_{k-1}\ell'$, and $\ell\leq_0\ell'$ implies $E(\ell)\subseteq E(\ell')$, the set enumerated by Worker $n$ is the same as that enumerated by Worker $1$. Information is passed from level $n$ down to level $1$ in a way that preserves the full history. Suppose Worker $n$ has produced $\sigma\ell^nu$. For $1\leq k\leq n-1$, Worker $k$ passes this information down, adding to the history its own most recent $\ell^k$, which satisfies $\ell^{k+1}\leq_k\ell^k$. Eventually, Worker $1$, having determined $\ell^1$, will see the top level sequence $\sigma\ell^nu$ together with the history 
		\[ \ell^n \leq_{n-1} \ell^{n-1} \leq_{n-2} \cdots \leq_2 \ell^2 \leq_1 \ell^1.\]
		Worker $1$ must find $\ell$ such that $\sigma\ell^nu\ell$ is in $P$, and for $1\leq k < n$, $\ell^k \leq_{k-1} \ell$. Ash's conditions guarantee that such an $\ell$ exists. Worker $1$ can find one since $P$ and the back-and-forth relations are all c.e. 
		
		\subsubsection{What is new in our proof?}
		
		We indicate what is the same and what is different in our proof. For simplicity, we treat the oracle of Worker $1$ as computable, although it is actually $\Delta^0_n(T)$. Our objects satisfy the hypotheses of Ash's metatheorem, except that the back-and-forth relations are not all c.e. Our instruction function gives a generic path through $T$, and $\cup_s E(\ell^n_s)$ is the atomic diagram of a structure $\mathcal{B}$ isomorphic to $\mathcal{A}^p$. Our back-and-forth relations are the symmetric ones, but they satisfy Ash's four conditions. If we try to proceed exactly as in the proof from \cite{AK00}, then as above, we would have $\sigma\ell^n u$, produced by Worker $n$, and $\ell^k$ for $1\leq k < n$, produced by Worker $k$, such that 
		\[ \ell^n \equiv_{n-1} \ell^{n-1} \equiv_{n-2} \cdots \equiv_2 \ell^2 \equiv_1 \ell^1.\]
		Since Ash's four conditions are satisfied, there is some $\ell$ such that $\sigma\ell^nu\ell\in P$, and $\ell^k \equiv_{k-1} \ell$ for $1\leq k<n$. The difficulty is that Worker $1$ needs to find $\ell$, and the relations $\equiv_k$ are not all c.e. 
		
		The new idea in our proof is that, instead of just passing down the instruction $u$, adding to the history, and leaving to Worker $1$ the responsibility of determining $\ell$, the workers modify the instruction.  The relations $\equiv_k$ and $\subseteq$ are defined on all steps---including the modified instructions, which we call \emph{requests}.  Worker $n$ transforms the instruction $u$ into a request $u^n$.  Worker $n-1$, having arrived at step $\ell^{n-1}$ before correctly and stably guessing the steps of Worker $n$ up to $\ell^nu^n$, produces the modified request $u^{n-1}\supseteq u^n$.  In general, for $2\leq k \leq n-1$, Worker $k$, having arrived at $\ell^k$ before correctly and stably guessing the sequence of steps of Worker $k+1$ up to $u^{k+1}$, modifies it to produce $u^k\supseteq u^{k+1}$. For $\ell$ to fulfill a request $u^k$ means that $\ell\supseteq u^k$.  Worker $1$, having correctly and stably guessed the steps of Worker $2$ up to $\ell^2u^2$, finds $\ell$ fulfilling the request $u^2$, and this also fulfills the requests $u^k$ for $2 < k\leq n$.    
		
In modifying the request $u^n$, Worker $n-1$ uses just $\equiv_{n-2}$ and $\subseteq$.  For $2 \leq k \leq n-2$, Worker $k$ determines $u^k$ using just $\equiv_{k-1}$ and $\equiv_k$.  Worker $1$ does not issue requests, but uses $\equiv_0$ and $\equiv_1$ in fulfilling the request $u^2$.\footnote{In the case $n=2$, Worker $1$ and Worker $n-1$ are the same. In this case, Worker $1$ uses $\equiv_0$ and $\subseteq$.  Several aspects of the proof below must be modified similarly in the case $n = 2$.}  This is the essential innovation that allows our construction to work.  The symmetric back-and-forth relations were forced on us because of the requests.  To indicate why this is, we consider how Worker $k$ determines a request.  Say that Worker $k$ has obtained $\ell^k$ and guessed $\ell^{k+1}u^{k+1}$.  We will have $\ell^{k+1}\equiv_k\ell^k$ and $\ell^{k+1}\equiv_k u^{k+1}$.  The request $u^k$ is determined in two steps.  First, find $\ell^*\supseteq \ell^{k+1}$ such that $\ell^k\equiv_{k-1} \ell^*$.  Then find $u^k\supseteq u^{k+1}$ such that $\ell^*\equiv_{k-1} u^k$.  We will have $\ell^k\equiv_{k-1} u^k$.  If we try to use the asymetric back-and-forth relations, with $\ell^{k+1}\leq_k\ell^k$ and $\ell^{k+1}\leq_k u^{k+1}$, we get stuck trying to find $u^k\supseteq u^{k+1}$ such that $\ell^k\leq_{k-1} u^k$. 
		
		\subsubsection{What does each worker do?}
		
		What we said above was very general. We must say more precisely what our workers do, given $T\cong T_\mathcal{A}$ for $\mathcal{A}$ of \ELL-rank at most $n$, to produce $\mathcal{B}\cong \mathcal{A}$. Worker~$1$ enumerates the diagram of $\mathcal{B}$. Worker $n$ produces an \ELL-$n$-generic path $p$ through $T$ and an isomorphism $F$ from $\mathcal{B}$ onto $\mathcal{A}^p$. Worker $n$ satisfies the following witnessing requirements.

\bigskip
\noindent
\textbf{Requirement $W_{(\nu,\tau)}$}:  For $\nu\prec\tau$ in $T$, where $\nu$ has variables $\bar{x}$ and $\tau$ has variables $\bar{x},\bar{y}$, if $\nu\prec p$, then $\nu$ extends to some $\tau'\prec p$ with variables including $\bar{x}$ and some $\bar{z}$ such that $(\tau,\bar{x},\bar{y})\equiv_{n-1}(\tau',\bar{x},\bar{z})$. 
		
\bigskip
Recall that the universe of $\mathcal{A}^p$ is the set $X = \{x_1,x_2,\ldots\ \}$. The atomic diagram of $\mathcal{A}^p$ is given by the labels on the nodes of the path $p$. The structure $\mathcal{B}$ is $\Delta^0_n$, while the path $p$ and the isomorphism $F$ are $\Delta^0_{2n-1}$ (all uniformly in $T$). Each worker produces a sequence of \emph{completed steps} of the form $\ell = (\sigma,g)$, where $\sigma$ is a node of $T$ that might lie on the path $p$, and $g$ is a finite $1-1$ function that might be a part of $F$, so that $g$ maps a tuple $\bar{b}$ in the universe of $\mathcal{B}$ to the variables of $\sigma$. If $\sigma$ is at level $k$ of $T$, then the variables are $x_1,\ldots,x_k$. The label on $\sigma$ assigns an atomic type to $\bar{x}$, and, therefore, to $\bar{b}$. The node $\sigma$, when thought of as part of a generic path $p$, also determines higher-complexity types. 
		
		\subsubsection{The c.e.\ set $\mathbb{L}$}
		
		Let $\mathbb{L}$ be the set of pairs $\ell = (\sigma,g)$ such that $\sigma\in T$ and $g$ is a finite $1-1$ function mapping some tuple $\bar{b}$ in $\omega$ (the universe of $\mathcal{B}$) onto the variables $\bar{x}$ of $\sigma$. The elements of $\mathbb{L}$ are the possible completed steps. There is a natural notion of extension on $\mathbb{L}$. For $\ell = (\sigma,g)$ and $\ell' = (\sigma',g')$, we write $\ell\subseteq\ell'$ if $\sigma\preceq\sigma'$ and $g\subseteq g'$. Below, we define back-and-forth relations $\ell\equiv_k\ell'$ on $\mathbb{L}$, with the intended meaning that the $\equiv_k$-class assigned by $\ell$ to a tuple $\bar{b}$ in the universe of $\mathcal{B}$ matches that assigned by $\ell'$. We have already defined relations $(\sigma,\bar{x})\equiv_k(\tau,\bar{y})$, where $\sigma,\tau\in T$ and $\bar{x}$, $\bar{y}$ are tuples of variables of the same length, with $\bar{x}$ included among the variables of $\sigma$ and $\bar{y}$ included among the variables of $\tau$. We use these relations to define relations $\equiv_k$ on $\mathbb{L}$. 
		
		\begin{defn}\label{def-equiv-g}
			
			Let $\ell = (\sigma,g)$ and $\ell' = (\sigma',g')$ be elements of $\mathbb{L}$ such that $\dom(g)\subseteq \dom(g')$. Say that $\dom(g) = \bar{b}$, $g(\bar{b}) = \bar{x}$, and $g'(\bar{b}) = \bar{y}$. 
			Then $\ell\equiv_k\ell'$ if $(\sigma,\bar{x})\equiv_k (\sigma',\bar{y})$. 
			
		\end{defn}
		
		\noindent
		\textbf{Remark}: Despite the notation, the relation $\equiv_k$ is not symmetric; the right-hand side may involve a function with strictly larger domain. 
		
		\subsubsection{Issuing and fulfilling requests}
		
		Worker $1$ produces just the sequence of completed steps. For $2\leq k\leq n$, Worker $k$ alternates completed steps with ``requests,'' of the form $(\tau,h,\bar{v})$, where $\tau\in T$, $h$ is a finite $1-1$ function taking a tuple $\bar{b}$ from $\omega$ to a tuple $\bar{x}$ among the variables of $\tau$, and $\bar{v}$ consists of the remaining variables of $\tau$. In issuing the request, Worker $k$ is asking the worker below for help in finding appropriate pre-images for the variables in $\bar{v}$. Let $\mathbb{U}$ be the set of these possible requests. 
We broaden our use of $\subseteq$ to include pairs from $\mathbb{U}\times\mathbb{U}$, and our use of the notations $\subseteq$, $\equiv_k$ to include pairs from $\mathbb{L}\times\mathbb{U}$ and $\mathbb{U}\times\mathbb{L}$. Suppose that $u = (\tau,h,\bar{v})$ and $u' = (\tau',h',\bar{v}')$ are in $\mathbb{U}$, and that $\ell = (\sigma,g)\in\mathbb{L}$. 
		
		\begin{enumerate}
			
			\item $u\subseteq u'$ if $\tau\preceq\tau'$ and $h\subseteq h'$.
			
			\item $\ell\subseteq u$ if $\sigma\preceq\tau$ and $g\subseteq h$, and for $k\leq n-1$, $\ell\equiv_k u$ if $\dom(g)\subseteq \dom(h)$, where $g$ maps $\bar{b}$ to $\bar{x}$, $h$ maps $\bar{b},\bar{b}'$ to $\bar{y},\bar{y}'$ such that $(\sigma,\bar{x})\equiv_k (\tau,\bar{y})$. 
			
			\item $u\subseteq\ell$ if $\tau\preceq\sigma$ and $g$ extends $h$ so as to include $\bar{v}$ in $\ran(g)$. For $k < n$, $u\equiv_k\ell$ if $\dom(h)\subseteq \dom(g)$, where $h$ maps $\bar{b}$ to the tuple $\bar{x}$ consisting of variables of $\tau$ not in $\bar{v}$, and $g$ maps $\bar{b},\bar{b}'$ to $\bar{x}',\bar{v}'$ such that $(\tau,\bar{x},\bar{v})\equiv_k(\sigma,\bar{x}',\bar{v}')$. 
			
		\end{enumerate} 
		
		\begin{defn}
			
			The request $u = (\tau,h,\bar{v})\in\mathbb{U}$ is ``fulfilled'' by $\ell = (\sigma,g)\in\mathbb{L}$ if $u\subseteq\ell$, and $\ran(g)$ includes the full tuple of variables of $\sigma$. 
			
		\end{defn} 
		
		\smallskip
		
\noindent
		\textbf{Note}:  If $u\subseteq u'$ and $\ell$ fulfills $u'$, then $\ell$ also fulfills $u$.
		
		\subsubsection{Back-and-forth relations computed by a given worker} 
		
		\begin{fact}\label{fact}
			
			Worker $k$, for $1\leq k \leq n$, can compute $\equiv_{k-1}$. For $1\leq k \leq n-2$, Worker $k$ can also compute $\equiv_k$. All of the workers can compute $\subseteq$.
			
		\end{fact}
		
		\begin{proof}
			
			The relation $\equiv_k$ is $\Pi_{2k}$-definable, so it can be decided using oracle $\Delta^0_{2k+1}$. For $1\leq k\leq n$, Worker~$k$ has oracle $\Delta^0_{n+k-1}$. For Worker $k$ to compute $\equiv_{k-1}$ we need $2(k-1)+1\leq n+k-1$, i.e., $k\leq n$. For Worker $k$ to compute $\equiv_k$ we need $2k+1\leq n+k-1$, i.e., $k\leq n-2$.
		\end{proof}
		
		\begin{center}
			\begin{tabular}{|c|c|c|}
				\hline
				Worker & Oracle & Computes \\
				\hline
				$n$ & $\Delta^0_{2n-1}$ & $\subseteq$, $\equiv_{n-1}$, \ldots \\
				$n-1$ & $\Delta^0_{2n-2}$ & $\subseteq$, $\equiv_{n-2}$, \ldots \\
				$n-2$ & $\Delta^0_{2n-3}$ & $\subseteq$, $\equiv_{n-2}$, $\equiv_{n-3}$,\ldots \\
				$\vdots$ & $\vdots$ & $\vdots$ \\
				$k$ & $\Delta^0_{n+k-1}$ & $\subseteq$, $\equiv_k$, $\equiv_{k-1}$, \ldots \\
				$\vdots$ & $\vdots$ & $\vdots$ \\
				$2$ & $\Delta^0_{n+1}$ & $\subseteq$,$\equiv_2$, $\equiv_1$, $\equiv_0$ \\
				$1$ & $\Delta^0_n$ & $\subseteq$,$\equiv_1$, $\equiv_0$ \\
				\hline
			\end{tabular}
		\end{center}
		
		The above table assumes that $n$ is large enough to make the first two rows sensible.
		
		\subsubsection{Steps of each worker}
		
		The setup of the worker argument is much like that in the proof in \cite{AK00} of Ash's metatheorem. For $2\leq k\leq n$, Worker $k$ produces an alternating sequence 
		$\ell^k_0u^k_1\ell^k_1u^k_2\ldots $, where $\ell^k_m = (\sigma^k_m,g^k_m)$ is a completed step and $u^k_{m+1}$ is the next request. The first completed step is the same for all workers---$\ell^k_0 = (\lambda,\emptyset)$, where $\lambda$ is the root of $T$ and $\emptyset$ is the empty function.
		Worker $1$ produces just a sequence of completed steps $\ell^1_0\ell^1_1\ell^1_2\ldots$, where $\ell^1_m = (\sigma^1_m,g^1_m)$. It is Worker $1$ who builds the structure $\mathcal{B}$. Suppressing the bottom-level oracle, we say that $\mathcal{B}$ is computable---it is actually $\Delta^0_n(T)$. A completed step $\ell^1_s = (\sigma^1_s,g^1_s)$ determines a map $g^1_s$ sending a tuple $\bar{b}\in \mathcal{B}$ to the variables of $\sigma^1_s$, and $\sigma^1_s$ determines the atomic type of these. Worker $1$ enumerates the atomic diagram of $\mathcal{B}$. To maintain consistency, Worker $1$ must preserve $\equiv_0$ between successive completed steps.
		
		For each $k$, Worker $k+1$ obtains its completed steps by accepting certain completed steps of Worker $k$, so the completed steps of Worker $k+1$ form a subsequence of the completed steps of Worker $k$. The completed steps of Worker $n$, on top, are the ones of Worker $1$ that have been accepted all the way up. Worker $n$ controls the construction of an $\ELL$-$n$-generic path $p$ through $T$ and an isomorphism $F$ from $\mathcal{B}$ onto $\mathcal{A}^p$. Worker $n$ must preserve the extension relation between successive completed steps. Since Worker $n$ can compute $\equiv_{n-1}$, it can meet the genericity requirements. 
		
		Having a completed step $\ell_s^n = (\sigma^n_s,g^n_s)$, Worker $n$ determines an extension $\sigma$ of $\sigma^n_s$ satisfying the next ``genericity'' requirement. Since the structure $\mathcal{B}$ must be $\Delta^0_n$ relative to $T$, Worker $n$ needs help from the workers below. Worker $n$ makes a request of Worker $n-1$, who will make a request of Worker $n-2$, and so on. Eventually, Worker $1$ fulfills the request from Worker $2$, in a completed step. This completed step will then be accepted by Worker $2$. The same completed step fulfills the request from Worker $3$, who accepts it, and by Worker $4$, who also accepts it, and so on up until it is accepted by Worker $n$. 
		
		Worker $k$ cannot compute the sequence produced by Worker $k+1$. However, Worker $k$ can approximate this sequence, such that for each finite initial segment, the approximations are eventually correct. Worker $k$ will give some requests to Worker $k-1$ that are based on incorrect guesses, but for each request $u$ of Worker $k+1$, Worker $k$ eventually has a stable correct guess at the sequence of Worker $k+1$ that ends in $u$. Then Worker $k$ produces a request $u'\supseteq u$. Moreover, the steps of Worker $k$ based on incorrect guesses never go off the rails. While guessing, possibly incorrectly, a longer sequence of Worker $k+1$, Worker $k$ maintains consistency with the initial segments that are correct. 
		
		\subsubsection{What is preserved, and what is a ``reasonable'' guess?}
		
		Worker $n$ preserves $\subseteq$; we require $\ell^n_i\subseteq u^n_{i+1}\subseteq\ell^n_{i+1}$. For Worker $k$, where $2\leq k\leq n-1$, we require $\ell^k_m\equiv_{k-1}u^k_{m+1}\subseteq\ell^k_{m+1}$. For Worker $1$, we require $\ell^1_m\equiv_0\ell^1_{m+1}$. We will have an additional condition as well: If $\ell^n_i = (\sigma,g)$ and $u^n_{i+1} = (\tau,h,\bar{v})$, then $\dom(g) = \dom(h)$. For $2\leq k\leq n-1$, Worker $k$, once it has $\ell^k_m$, determines $u^k_{m+1}$ based on a guess at the sequence of steps of Worker $k+1$ up to the first unfulfilled request. Worker $1$, once it has $\ell^1_m$, determines $\ell^1_{m+1}$ based on a guess at the sequence of steps of Worker $2$ up to the first unfulfilled request. To avoid fatal mistakes, we require that the guesses be ``reasonable.'' 
		
		Suppose Worker $k$ has produced a sequence ending in $\ell^k_m$ and is guessing at the sequence produced by Worker $k+1$ up to the first unfulfilled request. The guess will have the form $\ell^{k+1}_0u^{k+1}_1\ell^{k+1}_1\ldots\ell^{k+1}_p u^{k+1}_{p+1}$, where $\ell^{k+1}_i = \ell^k_{m_i}$ for some $0 = m_0 < m_1 < \ldots < m_p\leq m$. This indicates that Worker $k$ believes it has fulfilled request $u^{k+1}_i$ in $\ell^k_{m_i}$, and, as a result, $\ell^k_{m_i}$ has been accepted as $\ell^{k+1}_i$.  For the guess to be ``reasonable,'' we require more.
		
		\begin{defn}\ 
			
			\begin{enumerate}
				
				\item For $k = n-1$ being \emph{reasonable} means that $\ell^n_i\subseteq\ell^{n-1}_j$ for $m_i\leq j\leq m$.  In addition, $\ell^n_p\subseteq u^n_{p+1}$, and if $\ell^n_i = (\sigma,g)$ and $u^n_{i+1} = (\tau,h,\bar{v})$, then $\dom(g) = \dom(h)$.
				
				\item For $1\leq k\leq n-2$ being \emph{reasonable} means that $\ell^{k+1}_i\equiv_k\ell^k_j$ for $m_i\leq j\leq m$. In addition, $\ell^{k+1}_p\equiv_k u^{k+1}_{p+1}$, and if $\ell^{k+1}_i = (\sigma,g)$ and $u^{k+1}_{i+1} = (\tau,h,\bar{v})$, then $\dom(g) = \dom(h)$.
				
			\end{enumerate}
			
		\end{defn} 
		
		\subsubsection{More on requests}
		
		We must say exactly how Worker $k$ defines $u^k_{m+1}$, for $2\leq k\leq n$.
		
		\begin{enumerate}
			
			\item Given $\ell^n_m = (\sigma,g)$, Worker $n$ locates the first
				witnessing requirement $W_{(\nu,\tau)}$ that needs attention at
				$\sigma$. Thus $\nu\preceq\sigma$, and, writing $\bar{x}$ for the
				variables of $\nu$ and $\bar{x},\bar{y}$ for those of $\tau$, there is
				no tuple $\bar{z}$ among the variables of $\sigma$ such that $(\tau,\bar{x},\bar{y})\equiv_{n-1}(\sigma,\bar{x},\bar{z})$.
				Worker $n$ finds an extension $\sigma'\succeq\sigma$, with variables
				including some $\bar{z}$, such that $(\tau,\bar{x},\bar{y})\equiv_{n-1}(\sigma',\bar{x},\bar{z})$.
				Let $\bar{v}$ consist of the variables of $\sigma'$ that are not in
				the range of $g$. Then $u^n_{m+1}=(\sigma',g,\bar{v})$.
			
			\item Given $\ell^{n-1}_m = (\sigma,g)$, Worker $n-1$ guesses the sequence of steps of Worker $n$ up to the first unfulfilled request, checking that the guess is reasonable. Say that this guess ends in $\ell^n_p u^n_{p+1}$, where $\ell^n_p = (\sigma',g')$ and $u^n_{p+1} = (\tau,h,\bar{v})$. By reasonableness, we have $\ell^n_p\subseteq\ell^{n-1}_m$ and $\ell^n_p\subseteq u^n_{p+1}$, and $\dom(g')=\dom(h)$. Say that $g'$ maps $\bar{b}$ to $\bar{x}$; then $h$ also maps $\bar{b}$ to $\bar{x}$. Say that $g$ maps $\bar{b},\bar{b}'$ to $\bar{x},\bar{y}$. Choose an extension $\tau'$ of $\tau$ and a function $h'\supseteq h$ with $\dom(h')=\dom(g)$ such that $h'$ maps $\bar{b},\bar{b}'$ to $\bar{x},\bar{z}$ and $(\sigma,\bar{x},\bar{y})\equiv_{n-2}(\tau',\bar{x},\bar{z})$. Let $\bar{v}'$ consist of the variables of $\tau'$ that are not in $\ran(h')$. Then the desired request is $u^{n-1}_{m+1}=(\tau',h',\bar{v}')$. \par (Why do such $\tau'$ and $h'$ exist? Recall that, relative to $\sigma$, $\bar{x}$ and $\bar{y}$ represent some elements of $\mathcal{A}$. Since $\sigma'\preceq\sigma$ and $\sigma'\preceq\tau$, relative to $\sigma$ and $\tau$ the tuple $\bar{x}$ represents the same elements. There is some extension of $\tau$ in which we find elements automorphic to those represented by $\bar{y}$ over the elements represented by $\bar{x}$. The node $\tau'$ is long enough to find such elements. Worker $n-1$ is not able to check automorphism equivalence, but it can check $\equiv_{n-2}$.)
			
			\item Given $\ell^k_m = (\sigma,g)$, where $2\leq k\leq n-2$, Worker $k$ guesses the sequence of steps of Worker $k+1$ up to the first unfulfilled request, checking that the guess is reasonable. Say that this guess ends in $\ell^{k+1}_p u^{k+1}_{p+1}$, where $\ell^{k+1}_p = (\sigma',g')$ and $u^{k+1}_{p+1} = (\tau,h,\bar{v})$. By reasonableness, we have $\ell^{k+1}_p\equiv_k\ell^k_m$ and $\ell^{k+1}_p\equiv_k u^{k+1}_{p+1}$, and $\dom(g')=\dom(h)$. Say that $g'$ maps $\bar{b}$ to $\bar{x}$, $h$ maps $\bar{b}$ to $\bar{y}$, and $g$ maps $\bar{b},\bar{b}'$ to $\bar{z},\bar{z}'$. Choose an extension $\tau'$ of $\tau$ and a function $h'\supseteq h$ with $\dom(h')=\dom(g)$ such that $h'$ maps $\bar{b},\bar{b}'$ to $\bar{y},\bar{y}'$ and $(\sigma,\bar{z},\bar{z}')\equiv_{k-1}(\tau',\bar{y},\bar{y}')$. Let $\bar{v}'$ consist of the variables of $\tau'$ that are not in $\ran(h')$. Then the desired request is $u^k_{m+1}=(\tau',h',\bar{v}')$. \par (Why do such $\tau'$ and $h'$ exist? By reasonableness, $(\sigma',\bar{x})\equiv_k(\sigma,\bar{z})$ and $(\sigma',\bar{x})\equiv_k(\tau,\bar{y})$, so $(\tau,\bar{y})\equiv_k(\sigma,\bar{z})$. One step of the back-and-forth relation gives an extension $\tau'$ of $\tau$ and a tuple $\bar{y}'$ such that $(\sigma,\bar{z},\bar{z}')\equiv_{k-1}(\tau',\bar{y},\bar{y}')$. Worker $k$ can check $\equiv_k$ and $\equiv_{k-1}$.)
			
			\item Given $\ell^1_m = (\sigma,g)$, Worker $1$ guesses the sequence of steps of Worker $2$ up to the first unfulfilled request, checking that the guess is reasonable. Say that this guess ends in $\ell^2_p u^2_{p+1}$, where $\ell^2_p = (\sigma',g')$ and $u^2_{p+1} = (\tau,h,\bar{v})$. By reasonableness, we have $\ell^2_p\equiv_1\ell^1_m$ and $\ell^2_p\equiv_1 u^2_{p+1}$, and $\dom(g')=\dom(h)$. Say that $g'$ maps $\bar{b}$ to $\bar{x}$, $h$ maps $\bar{b}$ to $\bar{y}$, and $g$ maps $\bar{b},\bar{b}'$ to $\bar{z},\bar{z}'$. Choose an extension $\tau'$ of $\tau$ and a function $h'\supseteq h$ with $\dom(h')=\dom(g)$ such that $h'$ maps $\bar{b},\bar{b}'$ to $\bar{y},\bar{y}'$ and $(\sigma,\bar{z},\bar{z}')\equiv_0(\tau',\bar{y},\bar{y}')$. Extend $\tau'$ to $\sigma^*$ and $h'$ to a function $g''$, if necessary, so that $\ran(g'')$ includes all of the variables of $\sigma^*$ and $\dom(g'')$ includes the first $m+1$ elements of $\omega$. Then $\ell^1_{m+1}=(\sigma^*,g'')$ fulfills $u^2_{p+1}$. (Worker $1$ does not issue a request.) \par (Why do such $\tau'$ and $h'$ exist? By reasonableness, $(\sigma',\bar{x})\equiv_1(\sigma,\bar{z})$ and $(\sigma',\bar{x})\equiv_1(\tau,\bar{y})$, so $(\tau,\bar{y})\equiv_1(\sigma,\bar{z})$. One step of the back-and-forth relation gives an extension $\tau'$ of $\tau$ and a tuple $\bar{y}'$ such that $(\sigma,\bar{z},\bar{z}')\equiv_0(\tau',\bar{y},\bar{y}')$. Worker $1$ can check $\equiv_1$ and $\equiv_0$.)
			
		\end{enumerate}
		
		After issuing a request $u^k_m$, Worker $k$ must wait until Worker $k-1$ has fulfilled it. Worker $k-1$ will produce a request, based on a correct stable guess at what Worker $k$ has done, and then wait for Worker $k-2$ to fulfill it. Worker $k-2$ will eventually guess what Worker $k-1$ has done and produce a request. The requests form a chain. Eventually, Worker $1$ has a stable correct guess at what Worker $2$ has done, and fulfills Worker $2$'s request. This also fulfills the requests of the workers above, so it is accepted as the next completed step for all.

			\subsubsection{Accepting completed steps}
			
			Suppose that Worker $k$ has so far produced completed steps
			$\ell^k_0,\ldots,\ell^k_m$, and that Worker $k+1$ has accepted the
			completed steps $\ell^{k+1}_0,\ldots,\ell^{k+1}_p$. Thus, for some
			$0=m_0<\cdots<m_p<m$, we have
			$\ell^{k+1}_i=\ell^k_{m_i}$ for $0\leq i\leq p$. Suppose that the
			current request of Worker $k+1$ is $u^{k+1}_{p+1}$. We describe when
			Worker $k+1$ accepts the as-yet-unaccepted completed step
			$\ell^k_m$.
			
			For $k\geq2$, Worker $k+1$ waits until it sees that, in issuing
			$u^k_m$, Worker $k$ used a correct and stable approximation to the oracle of
			Worker $k+1$; that is, a true stage in the sense of Lachlan. In particular, Worker $k$ must have correctly guessed
			the sequence
			$\ell^{k+1}_0u^{k+1}_1\ell^{k+1}_1\cdots
			\ell^{k+1}_pu^{k+1}_{p+1}$. The request produced from this guess
			satisfies $u^k_m\supseteq u^{k+1}_{p+1}$, and Worker $k$ accepts
			$\ell^k_m$ only when it fulfills $u^k_m$. Hence
			$\ell^k_m\supseteq u^k_m\supseteq u^{k+1}_{p+1}$, so
			$\ell^k_m$ fulfills the outstanding request of Worker $k+1$. At this
			point Worker $k+1$ accepts it, setting
			$\ell^{k+1}_{p+1}=\ell^k_m$. When $k=1$, the middle request is
			omitted: after a correct and settled approximation to the oracle of
			Worker $2$, Worker $1$ constructs its completed
			step so as to fulfill
			$u^2_{p+1}$ directly.
			
			Whenever Worker $k+1$ accepts a completed step of Worker $k$, Worker
			$k$ had a correct and settled approximation to the oracle of Worker
			$k+1$. Thus, from that point on, Worker $k$'s guess at Worker $k+1$'s
			completed steps will always begin with
			\[
			\ell^{k+1}_0u^{k+1}_1\ell^{k+1}_1\cdots
			u^{k+1}_{p+1}\ell^{k+1}_{p+1}.
			\]
			In particular, if $k\leq n-2$, then
			$\ell^{k+1}_{p+1}\equiv_k\ell^k_j$ for all $j\geq m$; if
			$k=n-1$, then $\ell^n_{p+1}\subseteq\ell^{n-1}_j$ for all
			$j\geq m$.

		\subsubsection{Filling the gaps and completing the proof} 
		
		We have described the construction somewhat informally, leaving a few gaps.  We now fill those gaps and show that the construction succeeds. We must show that for $1\leq k\leq n-1$, once Worker $k$ has determined $\ell^k_m$, there is always a reasonable guess---namely, the truth---on which Worker $k$ can base its next step. We must also show that whenever Worker $k$ acts on a reasonable guess, even if it was wrong, no harm is done.

			\begin{lem}\label{lem:true-state-reasonable}
				Let $1\leq k\leq n-1$, and suppose that Worker $k$ has determined its sequence of completed steps
				\[
				\ell^k_0,\ell^k_1,\ldots,\ell^k_m
				\]
				and the true current state of Worker
				$k+1$ is
				\[
				\ell^{k+1}_0u^{k+1}_1\ell^{k+1}_1\cdots
				\ell^{k+1}_pu^{k+1}_{p+1},
				\]
				where $\ell^{k+1}_i=\ell^k_{m_i}$ for
				$0=m_0<\cdots<m_p\leq m$.
				\begin{enumerate}
					\item The true current state of Worker $k+1$, through its
					first unfulfilled request, is a reasonable guess on which Worker
					$k$ can base its next action.
					\item Suppose that Worker $k$ acts on a reasonable guess.
					\begin{enumerate}
						\item If $k=1$, then the
						resulting completed step $\ell_{m+1}^1$ satisfies
						\[
						\ell^2_i\equiv_1\ell^1_{m+1}
						\quad\text{for every }i\leq p.
						\]
						\item If $2\leq k\leq n-2$, then the resulting request $u^k_{m+1}$ satisfies
						\[
						\ell^{k+1}_i\equiv_k u^k_{m+1}
						\quad\text{for every }i\leq p
						\]
						and thus the resulting completed step $\ell_{m+1}^k$ satisfies
						\[
						\ell^{k+1}_i\equiv_k \ell^k_{m+1}
						\quad\text{for every }i\leq p.
						\]
						\item If $k=n-1$, then the resulting request $u^{n-1}_{m+1}$ satisfies
						\[
						\ell^n_i\subseteq u^{n-1}_{m+1}
						\quad\text{for every }i\leq p
						\]
						and thus the resulting completed step $\ell_{m+1}^k$ satisfies
						\[
						\ell^n_i\subseteq \ell^{n-1}_{m+1}
						\quad\text{for every }i\leq p.
						\]
					\end{enumerate}
				\end{enumerate} 
			\end{lem}
			
			\begin{proof}
				The proof is by simultaneous induction on the steps of the
				construction. At the start of the construction, all of these conditions hold. Now suppose that we are at some later point in the construction.
				
				For (1), when $k \leq n-2$, the preservation condition for Worker
				$k+1$ gives $\ell^{k+1}_p\equiv_k u^{k+1}_{p+1}$, while (2) gives, for each
				$i\leq p$ and each $j$ with $m_i\leq j\leq m$, $\ell^{k+1}_i\equiv_k\ell^k_j$. The
				functions in $\ell^{k+1}_i$ and $u^{k+1}_{i+1}$ have the same
				domain for every $i\leq p$ because of the way Worker $k+1$ issues its
				requests. These are exactly the conditions which say that the current true state is reasonable. If $k = n-1$ the argument is the same except that we use $\subseteq$ instead of $\equiv_k$.
				
				For (2), suppose that Worker $k$ acts on a reasonable guess. When Worker $k+1$ accepted the completed true stages through
				$\ell^{k+1}_p$, it was because Worker $k$ used correct and
				settled approximations to the oracle of Worker $k+1$. Thus every
				later approximation by Worker $k$ extends the settled true prefix
				\[
				\ell^{k+1}_0u^{k+1}_1 \cdots \ell^{k+1}_p.
				\]
				In particular, any later reasonable guess has the form
				\[
				\ell^{k+1}_0 \cdots \ell^{k+1}_p \hat{u}^{k+1}_{p+1} \hat{\ell}^{k+1}_{p+1}\cdots
				\hat{\ell}^{k+1}_{p+q}
				\hat{u}^{k+1}_{p+q+1}
				\]
				for some $q\geq 0$.
				By reasonableness, for every $i\leq p$,
				\[
				\ell^{k+1}_i\equiv_k \hat{\ell}_{p+q}^{k+1} \equiv_k \hat{u}^{k+1}_{p+q+1}
				\]
				when $k\leq n-2$, and
				\[
				\ell^n_i\subseteq \hat{\ell}_{p+q}^{k+1} \subseteq \hat{u}^n_{p+q+1}
				\]
				when $k=n-1$.
				We check the three cases.
				\begin{enumerate}
					\renewcommand{\labelenumi}{(\alph{enumi})}
					\item If $k=1$, Worker~$1$ constructs
					$\ell^1_{m+1}$ so as to fulfill $\hat{u}^2_{p+q+1}$. For every $i\leq p$,
					$\ell^2_i\equiv_1 \hat{u}^2_{p+q+1} \equiv_1 \ell^1_{m+1}$.
					\item If $2\leq k\leq n-2$, Worker $k$ constructs
					$u^k_{m+1}\supseteq \hat{u}^{k+1}_{p+q+1}$. Then
					$\ell^{k+1}_i\equiv_k u^k_{m+1}$ for every $i\leq p$.
					When $u^k_{m+1}$ is fulfilled by $\ell_{m+1}^k \supseteq u^k_{m+1}$ we have
					\[
					\ell^{k+1}_i\equiv_k\ell^k_{m+1}
					\quad\text{for every }i\leq p.
					\]
					\item Finally, if $k=n-1$, Worker $n-1$ constructs
					$u^{n-1}_{m+1}\supseteq \hat{u}^n_{p+q+1}$. By transitivity,
					\[
					\ell^n_i\subseteq u^{n-1}_{m+1}
					\quad\text{for every }i\leq p.
					\]
					When this request is fulfilled by $\ell_{m+1}^{n-1} \supseteq u_{m+1}^{n-1}$, transitivity again gives
					\[
					\ell^n_i\subseteq\ell^{n-1}_{m+1}
					\quad\text{for every }i\leq p.
					\]
				\end{enumerate}
				This gives (2).
			\end{proof}
			
			We can now argue that the construction works. We first verify that requests are passed down. We claim that, for
			$2\leq k\leq n$, whenever $\ell^k_m$ is defined, $u^k_{m+1}$ is
			defined. We argue by downward induction on $k$. For $k=n$, Worker $n$ always constructs a request. Suppose that
			$2\leq k<n$ and that the claim holds for Worker $k+1$. If
			$\ell^k_m$ is defined and $\ell^{k+1}_p$ is the true current completed
			step of Worker $k+1$, then $u^{k+1}_{p+1}$ is defined by the induction
			hypothesis. By Lemma~\ref{lem:true-state-reasonable}(1), the true
			current state of Worker $k+1$ through this request is reasonable.
			Hence Worker $k$ finds a reasonable guess and defines $u^k_{m+1}$.
			
			We next argue that whenever Worker $k+1$ makes a request
			$u^{k+1}_{p+1}$, Worker $k$ eventually fulfills it and Worker $k+1$
			accepts the resulting completed step. We prove this by induction on
			$k$. For the base case, fix a request $u^2_{p+1}$ of Worker $2$. At some point, Worker 1 has a correct and settled guess at the oracle of Worker $2$ through $u^2_{p+1}$. This guess is
			reasonable by Lemma~\ref{lem:true-state-reasonable}(1), so Worker $1$
			constructs a completed step that fulfills $u^2_{p+1}$. Because the
			guess was correct and settled, Worker $2$ accepts this completed step.
			
			Now suppose that $2\leq k\leq n-1$, and inductively that every
			request made by Worker $k$ is eventually fulfilled by Worker $k-1$
			and accepted by Worker $k$. In particular, whenever Worker $k$ issues
			a request, it eventually receives and accepts a completed step and can
			continue to act. Fix a request $u^{k+1}_{p+1}$ of Worker $k+1$. At some point, Worker $k$ has a correct and stable approximation of the state of Worker $k+1$. By
			Lemma~\ref{lem:true-state-reasonable}(1), this is a reasonable guess,
			and so Worker $k$ produces a request $u^k_{m+1} \supseteq u^{k+1}_{p+1}$. By the induction hypothesis, Worker $k-1$ eventually
			fulfills $u^k_{m+1}$. The resulting completed step of Worker $k$
			therefore also fulfills $u^{k+1}_{p+1}$. Since $u^k_{m+1}$ was issued
			from a correct and settled approximation to Worker $k+1$, Worker
			$k+1$ accepts that completed step.
			
			Thus the workers always make requests (other than Worker $1$) and eventually accept completed steps fulfilling those requests. In particular $\ell_m^k$ is defined for all $1 \leq k \leq n$ and all $m$. The sequence $\ell^1_m = (\sigma^1_m,g^1_m)$ is defined for all $m$, preserving $\equiv_0$, $\dom(g^1_m)$ is a tuple $\bar{b}$ that includes the first $m$ elements of $\omega$, and the label on $\sigma^1_m$ provides the atomic type of $\bar{b}$. Thus, Worker $1$ produces a well-defined structure $\mathcal{B}$ with domain equal to $\omega$. We have also seen that $\ell^n_m = (\sigma^n_m,g^n_m)$ is defined for all $m$, preserving $\subseteq$, and that for some $0 = s_0 < s_1 < \ldots$, $\ell^n_m = \ell^1_{s_m}$. Thus, Worker $n$ produces a path $p$ through $T$ and a function $F$ that is an isomorphism from $\mathcal{B}$ onto $\mathcal{A}^p$. Worker $n$ ensures that this path is $\ELL$-$n$-generic.
			
			We have described a uniform effective procedure which, for all $\mathcal{A}$ of $\ELL$-rank at most $n$ and $T\cong T_\mathcal{A}$, produces a structure $\mathcal{B}$, computable in $\Delta^0_n(T)$, such that $\mathcal{B}\cong\mathcal{A}^p$ for some $\ELL$-$n$-generic path $p$ through $T$.
	\end{proof} 
	
	\subsection{The case $\alpha$ a limit ordinal}
	
	In this subsection, we prove Theorem 1.1 in Case 3, where $\alpha$ is a limit ordinal.
	
	\begin{case}[$\alpha$ a limit ordinal]
		For each computable limit ordinal $\alpha$, there is an effective $\Delta^0_\alpha$ operator that, for all $\mathcal{A}$ of \ELL-rank at most $\alpha$, takes copies of $T_\mathcal{A}$ to copies of $\mathcal{A}$. 
	\end{case}
	
	\begin{proof}
		
		For $\mathcal{A}$ of \ELL-rank at most $\alpha$ and $T \cong T_{\mathcal{A}}$, we can apply a uniform $\Delta^0_\alpha$ procedure to find an \ELL-$\alpha$-generic path $p$ through $T$ (see Remark \ref{rmk:getGeneric}). From $p$, we pass effectively to $\mathcal{A}^p$. If $T\cong T_\mathcal{A}$, then this $\mathcal{A}^p$ is isomorphic to $\mathcal{A}$. 
	\end{proof} 
	
	\subsection{$\alpha = \lambda + n$ for $\lambda$ a limit ordinal and $n \geq 1$}
	
	In this subsection, we complete the proof of Theorem 1.1, with Case 4, where $\alpha = \lambda+n$ for $\lambda$ a limit ordinal and $n$ finite, with $n \geq 1$.
	
	\begin{case}
		If $\alpha = \lambda+n$, where $\lambda$ is a computable limit ordinal and $n\geq 1$ is finite, then there is an effective $\Delta^0_\alpha$ operator that, for $\mathcal{A}$ of \ELL-rank at most $\alpha$, takes each $T\cong T_\mathcal{A}$ to a copy of $\mathcal{A}$.
	\end{case}
	
	\begin{proof} [Sketch]
		
		We imitate the proof of Case 1 if $n=1$, and of Case 2 if $n\geq 2$. The imitation of Case 1 is straightforward. For Case 2, we use $n$ workers. For $1\leq k\leq n$, Worker $k$ has oracle $\Delta^0_k(\Delta^0_{\lambda+n}(T)) = \Delta^0_{\lambda+n+k-1}(T)$. This ensures that the analog of Fact \ref{fact} holds. Worker $n$ produces an \ELL-$\alpha$-generic path $p$ and an isomorphism $F$ from $\mathcal{B}$ onto $\mathcal{A}^p$. The important back-and-forth relations are between pairs $(\sigma,\bar{x})$, where $\sigma\in T$ and $\bar{x}$ is among the variables of $\sigma$. The completed steps of Worker $k$ have form $\ell^k_s = (\sigma^k_s,g^k_s)$, where $g^k_s$ maps a tuple from $\mathcal{B}$ to the variables $\bar{x}^k_s$ of $\sigma^k_s$. Worker $k$ preserves $\equiv_{\lambda + k-1}$; that is, for $\bar{x}$ included among the variables of $\sigma^k_s$, $(\sigma^k_s,\bar{x})\equiv_{\lambda + k-1}(\sigma^k_{s+1},\bar{x})$. 
		Worker $1$ builds $\mathcal{B}$. Worker $1$ makes sure that the first $s$ elements of $\mathcal{B}$ are included in $\dom(g^1_s)$. To determine the atomic type of the tuple consisting of the first $s$ elements of $\mathcal{B}$, Worker $1$ simply looks at the label on $\sigma^1_s$. 
	\end{proof}
	
	\section{Hardness}
	
	By relativizing the main result of the previous section, we obtain that for each $X$-computable ordinal $\alpha\geq 1$, there is an $X$-effective $\Delta^0_\alpha$ operator that for all $\mathcal{A}$ of Scott rank at most $\alpha$, takes copies of $T_\mathcal{A}$ to copies of $\mathcal{A}$. We believe that this is optimal; that is, for an ordinal $\alpha\geq 2$, and $1\leq \beta < \alpha$, there is no (boldface) $\bfDelta^0_\beta$ operator that, for all $\mathcal{A}$ of Scott rank at most $\alpha$, takes copies of $T_\mathcal{A}$ to copies of $\mathcal{A}$. The only method we know for showing non-existence of the $\bfDelta^0_\beta$ operator is, fixing a parameter $X$ making the operator $X$-effective $\Delta^0_\beta$, to give a structure $\mathcal{A}$ of Scott rank at most $\alpha$ such that $T_\mathcal{A}$ has an $X$-computable copy while $\mathcal{A}$ has no copy $\Delta^0_\beta$ relative to $X$. We can carry this out for finite levels.
	
	\begin{thm}
		For each $X \subseteq \mathbb{N}$ and $n \in \mathbb{N}$, there is a structure $\mathcal{A}$ of Scott rank $n+2$ such that $T_\mathcal{A}$ has an $X$-computable copy but $\mathcal{A}$ has no $X^{(n)}$-computable copy. Thus, there is no (boldface) $\bfDelta^0_{n+1}$ operator taking copies of $T_\mathcal{A}$ to copies of $\mathcal{A}$ working uniformly on all structures $\mathcal{A}$ of Scott rank at most $n+2$. 
	\end{thm}
	
	Working relative to $X$ everywhere, we can suppress $X$. The proof is essentially just putting together two ingredients. Harrison-Trainor and Montalb\'{a}n \cite{HTMontalban} gave an example of a structure $\mathcal{B}$ such that $T_\mathcal{B}$ has a computable copy, but $\mathcal{B}$ has no computable copy. On inspection, this structure has Scott rank $2$. The construction can be relativized to give a structure $\mathcal{B}$ such that $T_\mathcal{B}$ has a $0^{(2n)}$-computable copy but $\mathcal{B}$ has no $0^{(2n)}$-computable copy. For simplicity we may assume that $\mathcal{B}$ is a graph.
	
	The other ingredient is the finite unfriendly jump inversions introduced by Harrison-Trainor in \cite{unfriendly}. From the structure $\mathcal{B}$, we can construct the $n$th unfriendly jump inversion $\Inv{n}{\mathcal{B}}$ of $\mathcal{B}$. We will take $\mathcal{A} = \Inv{n}{\mathcal{B}}$. We will need the formal definition only to prove one fact below.
	
	\begin{defn}[Harrison-Trainor]
		Let $S \subseteq \mathbb{N}$ be a complete $\Sigma^0_{2n+1}$ set. We fix computable structures $\mathcal{M} \ncong \mathcal{N}$ and a uniformly computable sequence of computable structures $\mathcal{C}_i$ such that
		\begin{enumerate}
			\item For each $i$,
			\[\mathcal{C}_i\cong\left\{\begin{array}{cc} \mathcal{M} & \mbox{if $i\in S$}\\
				\mathcal{N} & \mbox{otherwise}
			\end{array}\right.\]
			\item $\mathcal{N} \nleq_{n+1} \mathcal{M}$.
			\item $\mathcal{M}$ and $\mathcal{N}$ have $\Pi_{n+2}$ Scott sentences.
		\end{enumerate}
		Such structures exist, as shown in \cite{unfriendly}.
		Let $\mathcal{G}$ be a graph.\footnote{We could, with simple modifications, deal with a structure in any countable relational language.} Define $\Inv{n}{\mathcal{G}}$ to be the structure consisting of the vertices of $\mathcal{G}$, and associated to each unordered pair $\{u,v\}$ of distinct vertices, an ordered pair of structures isomorphic to $(\mathcal{M},\mathcal{N})$ if there is an edge $u - v$, and isomorphic to $(\mathcal{N},\mathcal{M})$ if there is no edge $u - v$. More formally, let $\mathcal{L}$ be the language of $\mathcal{M}$ and $\mathcal{N}$. We work in the language $\mathcal{L}^* = \mathcal{L} \cup \{U,E,F\}$ where $U$ is a unary relation and $E$ and $F$ are ternary relations. $\Inv{n}{\mathcal{G}}$ will be the structure defined as follows. The domain of $\Inv{n}{\mathcal{G}}$ will consist of the elements of $\mathcal{G}$ together with, for each pair of distinct $u,v \in \mathcal{G}$, disjoint infinite sets $C_{\{u,v\}}$ and $D_{\{u,v\}}$. $U$ will hold only of the elements of $\mathcal{G}$. For distinct $u,v$, we put $E(u,v,x)$ if and only if $x \in C_{\{u,v\}}$ and $F(u,v,x)$ if and only if $x \in D_{\{u,v\}}$; otherwise, $E$ and $F$ do not hold. The relations of $\mathcal{L}$ will hold only of tuples of elements all in the same set $C_{\{u,v\}}$ or $D_{\{u,v\}}$, and we will have $C_{\{u,v\}} \cong \mathcal{M}$ and $D_{\{u,v\}} \cong \mathcal{N}$ if there is an edge connecting $u$ and $v$, and $C_{\{u,v\}} \cong \mathcal{N}$ and $D_{\{u,v\}} \cong \mathcal{M}$ if there is no such edge.
	\end{defn}
	
	We will use the following properties of this jump inversion.
	
	\begin{thm}[Harrison-Trainor \cite{unfriendly}]\label{thm:unfriendly}{\ }
		\begin{enumerate}
			\item Given a $\mathbf{0}^{(2n)}$-computable copy of $\mathcal{B}$, there is a computable copy of $\Inv{n}{\mathcal{B}}$.
			\item Given a $\mathbf{0}^{(n)}$-computable copy of $\Inv{n}{\mathcal{B}}$, there is a $\mathbf{0}^{(2n)}$-computable copy of $\mathcal{B}$.
			\item If $\mathcal{B}$ has a $\Pi_\ell$ Scott sentence for $\ell \geq 2$, then $\Inv{n}{\mathcal{B}}$ has a $\Pi_{n+\ell}$ Scott sentence.
		\end{enumerate}
	\end{thm}
	
	In particular, since $\mathcal{B}$ has Scott rank $2$; i.e., it has a $\Pi_3$ Scott sentence, $\mathcal{A} = \Inv{n}{\mathcal{B}}$ has a $\Pi_{n+3}$ Scott sentence and hence has Scott rank at most $n+2$.
	If $\mathcal{A} = \Inv{n}{\mathcal{B}}$ had a $0^{(n)}$-computable copy, then $\mathcal{B}$ would have a $0^{(2n)}$-computable copy; since $\mathcal{B}$ does not have a $0^{(2n)}$-computable copy, $\mathcal{A} =\Inv{n}{\mathcal{B}}$ does not have a $0^{(n)}$-computable copy.
	
	Finally, we have a $0^{(2n)}$-computable copy of $T_{\mathcal{B}}$ and must argue that there is a computable copy of $T_{\mathcal{A}} = T_{\Inv{n}{\mathcal{B}}}$. This is essentially the unfriendly analogue of Lemma 6.3 of \cite{HTMontalban}; intuitively, $T_{\Inv{n}{\mathcal{B}}}$ and $\Inv{n}{T_{\mathcal{B}}}$ are not too far apart, and (1) of Theorem \ref{thm:unfriendly} implies that $T_{\mathcal{A}}$ (which is morally $\Inv{n}{T_{\mathcal{B}}}$) has a computable copy. The following lemma makes this formal.
	
	\begin{lem}
		For any structure $\mathcal{B}$, if $T_{\mathcal{B}}$ has a $0^{(2n)}$-computable copy then $T_{\Inv{n}{\mathcal{B}}}$ has a computable copy.
	\end{lem}
	\begin{proof}
		Assume for simplicity that $\mathcal{B}$ is a graph. Construct computable sequences of structures $\mathcal{C}_{\sigma,\tau}$ and $\mathcal{D}_{\sigma,\tau}$, where $\sigma, \tau \in T_{\mathcal{B}}$ are compatible nodes, such that $(\mathcal{C}_{\sigma,\tau},\mathcal{D}_{\sigma,\tau}) \cong (\mathcal{M},\mathcal{N})$ if in $T_{\mathcal{B}}$ there is an edge between (the elements of $\mathcal{B}$ represented by the last entries of) $\sigma$ and $\tau$, and $(\mathcal{C}_{\sigma,\tau},\mathcal{D}_{\sigma,\tau}) \cong (\mathcal{N},\mathcal{M})$ otherwise. In this and what follows, we treat $\sigma,\tau$ as an unordered pair.
		
		The underlying tree of our computable copy of $T_{\Inv{n}{\mathcal{B}}}$ will be the tree of tuples $\langle b_1,\ldots,b_\ell \rangle$ where:
		\begin{enumerate}
			\item each $b_i$ is either of the form $\langle \sigma,e \rangle$ where $\sigma$ is a non-root node in $T_{\mathcal{B}}$ and $e \in \omega$, or is of the form $\langle \sigma,\tau; p, r ; e \rangle$ for $\sigma,\tau$ distinct compatible nodes in $T_{\mathcal{B}}$, $p \in \{0,1\}$, $e \in \omega$, and $r \in \mathcal{C}_{\sigma,\tau}$ if $p = 0$ and $r \in \mathcal{D}_{\sigma,\tau}$ if $p = 1$;
			\item all of the nodes in $T_{\mathcal{B}}$ appearing in $b_1,\ldots,b_\ell$, including the nodes $\sigma,\tau$ when $b_i = \langle \sigma,\tau; p, r ; e \rangle$, are from the same path in $T_{\mathcal{B}}$;
			\item among $b_1,\ldots,b_\ell$ there is no repetition of the form $\langle \sigma,e \rangle$ and $\langle \sigma,e' \rangle$ or $\langle \sigma,\tau; p, r ; e \rangle$ and $\langle \sigma,\tau; p, r ; e' \rangle$.
		\end{enumerate}
		The $b_i$ of the form $\langle \sigma,e \rangle$ represent elements of $\mathcal{B}$, and nodes $\langle \sigma,\tau; p, r ; e \rangle$ represent the $r$th element of one of the structures $\mathcal{C}_{\sigma,\tau}$ or $\mathcal{D}_{\sigma,\tau}$ attached to (the elements of $\mathcal{B}$ represented by the last entries of) $\sigma$ and $\tau$. Essentially, we have ``mixed in'' the elements of the $\mathcal{C}_{\sigma,\tau}$ and $\mathcal{D}_{\sigma,\tau}$. The $e$ are simply to make sure that everything appears infinitely many times as in the definition of the tree of tuples.
		
		Now, given a node $\langle b_1,\ldots,b_\ell \rangle$ of the proposed copy of $T_{\Inv{n}{\mathcal{B}}}$, we must computably determine the atomic diagram in the language $\mathcal{L}^* = \mathcal{L} \cup \{ U,E,F\}$ of the elements of ${\Inv{n}{\mathcal{B}}}$ they represent. This is easy to do. Given $b_{i_1},\ldots,b_{i_m}$ with $b_{i_t}$ of the form $\langle \sigma,\tau; p, r_t ; e_t \rangle$ (for fixed $\sigma,\tau,p$) we put the relations of $\mathcal{L}$ on them according to how they are on $r_1,\ldots,r_m$ in $\mathcal{C}_{\sigma,\tau}$ (if $p = 0$) or $\mathcal{D}_{\sigma,\tau}$ (if $p = 1$). We put $U$ only on those $b_i$ which are of the form $\langle \sigma, e \rangle$, representing an element of $\mathcal{B}$. And we put $E(b_i,b_j,b_k)$ only when $b_i = \langle \sigma, e \rangle$ and $b_j = \langle \tau, f \rangle$ are nodes in $T_{\mathcal{B}}$ and $b_k = \langle \sigma,\tau; 0,r ; g \rangle$ and $F(b_i,b_j,b_k)$ only when $b_i = \langle \sigma , e \rangle$ and $b_j = \langle \tau , f \rangle$ are nodes in $T_{\mathcal{B}}$ and $b_k = \langle \sigma,\tau;1,r; g  \rangle$.
	\end{proof}
	
	It follows that for all finite $n\geq 1$, there is no $\bfDelta^0_n$ operator that, for all $\mathcal{A}$ of Scott rank at most $\omega$, takes copies of $T_\mathcal{A}$ to copies of $\mathcal{A}$. Hence, the first case where we do not have a proof of sharpness is $\omega+1$. There is an effective $\Delta^0_{\omega+1}$ operator that, for all $\mathcal{A}$ of Scott rank at most $\omega+1$, takes copies of $T_\mathcal{A}$ to copies of $\mathcal{A}$, and we do not know whether there is a $\bfDelta^0_\omega$ operator. 
	
	\begin{question}
		Given $\alpha\geq\omega+1$, does there exist $\beta < \alpha$ and a $\bfDelta^0_\beta$ operator that takes copies of $T_\mathcal{A}$ to copies of $\mathcal{A}$ for all $\mathcal{A}$ with Scott rank at most $\alpha$?
	\end{question}
	
	Andrews, Gonzalez, and Zhu \cite{belligerent} have extended unfriendly jump inversions through the infinite ordinals. However, this is not sufficient for our purposes. The issue, as discussed in that paper, stems from the fact that for infinite $\alpha$, $2 \cdot \alpha \neq \alpha \cdot 2$.
	
	\bibliographystyle{alpha}
	\bibliography{Reference}

@article{Lavrov,
	author = {I. A. Lavrov},
	journal = {Algebra and Logic},
	pages = {5-18},
	title = {Effective inseparability of the set of identically true formulae and finitely refutable formulae for certain elementary theories},
	volume = {2},
	year = {1963}}

@unpublished{HTKim,
	author = {Matthew Harrison-Trainor and J. Thomas Kim},
	note = {Preprint available at arxiv.org/pdf/2602.07166},
	title = {Scott spectral gaps for trees are bounded}}

@incollection{Karp,
	author = {Karp, Carol R.},
	booktitle = {Theory of {M}odels ({P}roc. 1963 {I}nternat. {S}ympos. {B}erkeley)},
	mrclass = {02.35},
	mrnumber = {209132},
	mrreviewer = {J.\ E.\ Fenstad},
	pages = {407--412},
	publisher = {North-Holland, Amsterdam},
	title = {Finite-quantifier equivalence},
	year = {1965}}

@article{MR,
	author = {A. Montalb\'an and D. Rossegger},
	journal = {The Journal of Symbolic Logic},
	pages = {1-17},
	title = {The Structural Complexity of Models of Arithmetic},
	volume = {First View},
	year = {2023}}

@article{GLRS,
	author = {D. Gonzalez and M. Lelyk and D. Rossegger and P. Szlufik},
	journal = {Submitted for publication. https://arxiv.org/abs/2507.12025},
	title = {Classifying the Complexities of Models of Arithmetic}}

@article{CamerloGao,
	author = {Riccardo Camerlo and Su Gao},
	issn = {00029947, 10886850},
	journal = {Transactions of the American Mathematical Society},
	number = {2},
	pages = {491--518},
	publisher = {American Mathematical Society},
	title = {The Completeness of the Isomorphism Relation for Countable {B}oolean Algebras},
	url = {http://www.jstor.org/stable/221773},
	urldate = {2025-04-17},
	volume = {353},
	year = {2001}}

@article{FriedmanStanley,
	author = {Friedman, Harvey and Stanley, Lee},
	doi = {10.2307/2274750},
	fjournal = {The Journal of Symbolic Logic},
	issn = {0022-4812,1943-5886},
	journal = {J. Symbolic Logic},
	mrclass = {03C15 (03C75 03E15)},
	mrnumber = {1011177},
	number = {3},
	pages = {894--914},
	title = {A {B}orel reducibility theory for classes of countable structures},
	url = {https://doi.org/10.2307/2274750},
	volume = {54},
	year = {1989}}

@article{GonzalezRossegger,
	author = {Gonzalez, David and Rossegger, Dino},
	doi = {10.1017/jsl.2024.59},
	fjournal = {The Journal of Symbolic Logic},
	issn = {0022-4812,1943-5886},
	journal = {J. Symb. Log.},
	mrclass = {03C57 (03C70 03D45 03E15)},
	mrnumber = {5027361},
	number = {4},
	pages = {1563--1592},
	title = {Scott sentence complexities of linear orderings},
	url = {https://doi.org/10.1017/jsl.2024.59},
	volume = {90},
	year = {2025}}

@article{Gao,
	author = {Gao, Su},
	doi = {10.2307/2695052},
	fjournal = {The Journal of Symbolic Logic},
	issn = {0022-4812,1943-5886},
	journal = {J. Symbolic Logic},
	mrclass = {03C15 (03C64 03C75 03E15)},
	mrnumber = {1833486},
	mrreviewer = {Andreas\ Blass},
	number = {2},
	pages = {902--922},
	title = {Some dichotomy theorems for isomorphism relations of countable models},
	url = {https://doi.org/10.2307/2695052},
	volume = {66},
	year = {2001}}

@unpublished{errata,
	author = {Harrison-Trainor, Matthew and Montalb\'an, Antonio},
	note = {\url{https://homepages.math.uic.edu/~mht/papers/errata-tree-tuples.pdf}},
	title = {Errata to: The tree of tuples of a structure},
	year = {2026}}

@article{HTMontalban,
	author = {Harrison-Trainor, Matthew and Montalb\'an, Antonio},
	doi = {10.1017/jsl.2019.92},
	fjournal = {The Journal of Symbolic Logic},
	issn = {0022-4812,1943-5886},
	journal = {J. Symb. Log.},
	mrclass = {03D45 (03C57)},
	mrnumber = {4404618},
	mrreviewer = {Ivan\ V.\ Latkin},
	number = {1},
	pages = {21--46},
	title = {The tree of tuples of a structure},
	url = {https://doi.org/10.1017/jsl.2019.92},
	volume = {87},
	year = {2022}}

@article{Lauchli,
	author = {L\"auchli, H. and Leonard, J.},
	doi = {10.4064/fm-59-1-109-116},
	fjournal = {Polska Akademia Nauk. Fundamenta Mathematicae},
	issn = {0016-2736,1730-6329},
	journal = {Fund. Math.},
	mrclass = {02.54},
	mrnumber = {199108},
	mrreviewer = {E.\ Engeler},
	pages = {109--116},
	title = {On the elementary theory of linear order},
	url = {https://doi.org/10.4064/fm-59-1-109-116},
	volume = {59},
	year = {1966}}

@article{Ehrenfeucht,
	author = {Ehrenfeucht, A.},
	doi = {10.4064/fm-49-2-129-141},
	fjournal = {Polska Akademia Nauk. Fundamenta Mathematicae},
	issn = {0016-2736,1730-6329},
	journal = {Fund. Math.},
	mrclass = {02.50},
	mrnumber = {126370},
	mrreviewer = {G.\ F.\ Rose},
	pages = {129--141},
	title = {An application of games to the completeness problem for formalized theories},
	url = {https://doi.org/10.4064/fm-49-2-129-141},
	volume = {49},
	year = {1960/61}}

@book{AK00,
	author = {Ash, C.J. and Knight, J.},
	publisher = {Elsevier Science},
	title = {Computable Structures and the Hyperarithmetical Hierarchy},
	year = {2000}}

@article{robuster,
	author = {Montalb\'an, Antonio},
	doi = {10.1090/proc/12669},
	fjournal = {Proceedings of the American Mathematical Society},
	issn = {0002-9939,1088-6826},
	journal = {Proc. Amer. Math. Soc.},
	mrclass = {03D45 (03C57 03C75 03D60)},
	mrnumber = {3411157},
	mrreviewer = {Rodney\ G.\ Downey},
	number = {12},
	pages = {5427--5436},
	title = {A robuster {S}cott rank},
	url = {https://doi.org/10.1090/proc/12669},
	volume = {143},
	year = {2015}}

@unpublished{ChenGonzalezHT,
	author = {Ruiyan Chen and David Gonzalez and Matthew Harrison-Trainor},
	note = {Preprint, arXiv:2505.00893},
	title = {Optimal syntactic definitions of back-and-forth types}}

@incollection{Scott,
	author = {Scott, Dana},
	booktitle = {Theory of {M}odels ({P}roc. 1963 {I}nternat. {S}ympos. {B}erkeley)},
	mrclass = {02.35},
	mrnumber = {200133},
	mrreviewer = {E.\ Engeler},
	pages = {329--341},
	publisher = {North-Holland, Amsterdam},
	title = {Logic with denumerably long formulas and finite strings of quantifiers},
	year = {1965}}

@unpublished{GonzalezKnight,
	author = {David Gonzalez and Julia F. Knight},
	note = {Preprint, arXiv: 2605.04404},
	title = {Computable {S}cott sentences and the {F}riedman-{S}tanley embedding of graphs in labeled trees}}

@unpublished{unfriendly,
	author = {Matthew Harrison-Trainor},
	note = {Preprint, arXiv: 2505.23613},
	title = {Unfriendly jump inversions}}

@unpublished{belligerent,
	author = {Uri Andrews and David Gonzalez and Hongyu Zhu},
	note = {Preprint, arXiv:2607.10935},
	title = {Infinite Belligerent Jump Inversion and Computable Scott Analysis}}
	
\end{document}